\documentclass[a4paper, 11pt]{article}

\usepackage{amsmath}
\usepackage{amsfonts}
\usepackage{amsthm}
\usepackage{amssymb}
\usepackage{tikz}
\usetikzlibrary{positioning}
\usetikzlibrary{arrows.meta}
\usepackage{thmtools, thm-restate}
\usetikzlibrary{patterns}
\usepackage[a4paper, margin=2.5cm]{geometry}

\usepackage{hyperref}
\usepackage{bookmark}

\newcommand{\grad}{\ensuremath{\nabla}}

\newcommand{\esssup}{\operatorname{esssup}}
\newcommand{\suf}{surface function}
\newcommand{\R}{\ensuremath{\mathbb{R}}}
\newcommand{\N}{\ensuremath{\mathbb{N}}}
\newcommand{\dd}{\operatorname{d}}

\newcommand{\future}[1]{\ensuremath{\textrm{I}^{\hspace{0.5pt}+}\!\left({#1}\right)}}
\newcommand{\past}[1]{\ensuremath{\textrm{I}^{\hspace{0.5pt}-}\!\left({#1}\right)}}

\newcommand{\causalfuture}[1]{\ensuremath{\textrm{J}^+\!\left({#1}\right)}}
\newcommand{\causalpast}[1]{\ensuremath{\textrm{J}^-\!\left({#1}\right)}}
\newcommand{\closure}[1]{\ensuremath{\overline{{#1}}}}
\newcommand{\interior}[1]{\operatorname{int}\left({#1}\right)}
\newcommand{\missingf}[1]{\ensuremath{\text{Miss}^+\!\left({#1}\right)}}
\newcommand{\missingp}[1]{\ensuremath{\text{Miss}^-\!\left({#1}\right)}}
\newcommand{\missingfp}[1]{\ensuremath{\text{Miss}^\pm\!\left({#1}\right)}}

\newcommand{\len}[1]{\ensuremath{L\left({#1}\right)}}

\newcommand{\EmptySet}{\ensuremath{\varnothing}}
\newcommand{\domain}[1]{\ensuremath{\text{dom}({#1})}}

\newcommand{\paths}[2]{\ensuremath{\Omega_{{#1},{#2}}}}
\newcommand{\eikonal}{\ensuremath{\operatorname{\mathcal{S}}\!\left(M\right)}}
\newcommand{\abs}[1]{\ensuremath{\left\lvert{#1}\right\rvert}}

\newcommand{\tikzRightarrow}{%
    \begin{tikzpicture}[scale=0.5, baseline=-1.2mm]
      \draw[double, ->] (0,0) -- (8mm,0mm);
    \end{tikzpicture}%
}

\declaretheorem[numberwithin=section]{theorem}
\declaretheorem[sibling=theorem]{definition}
\declaretheorem[sibling=theorem]{lemma}
\declaretheorem[sibling=theorem]{corollary}
\declaretheorem[sibling=theorem]{proposition}

\theoremstyle{definition}
\newtheorem{examplex}[theorem]{Example}
\newenvironment{example}
{\pushQED{\qed}\examplex}
{\popQED\endexamplex}

\title{The global properties of the finiteness 
and continuity of the Lorentzian distance}

\author{Adam Rennie and Ben Whale\\
\\
School of Mathematics and Applied Statistics\\
University of Wollongong\\
Northfields Ave, Wollongong, NSW, 2522}

\begin{document}

\maketitle

\abstract{
  It is well-known that global hyperbolicity implies
  that the Lorentzian distance is finite and continuous.
  By carefully analysing the causes of discontinuity of the Lorentzian
  distance, we show that in most other respects the finiteness
  and continuity of the Lorentzian distance is independent of the 
  causal structure. The proof of these results relies on 
  the properties of a class of generalised time functions 
  introduced by the authors in \cite{RennieWhale2016}.
}

\medskip



\parindent=0.0in
\parskip=0.08in

\section{Introduction}
\label{section.introduction}

  It is well-known that the Lorentzian
  distance in a globally hyperbolic manifold is finite and continuous,
  \cite[Lemma 4.5]{BeemEhrlichEasley1996}.
  There are a handful of other results that describe the properties of
  manifolds with continuous Lorentzian distances, e.g.\
  \cite[Theorem 4.24]{BeemEhrlichEasley1996} and
  \cite[Theorems 2.2, 2.4 and 3.6]{Minguzzi2009}.
  These results suggest that the Lorentzian distance should be related, 
  at least in a conformal sense, to other conditions in the causal hierarchy.
  
  We show that this is not so, apart from the few results
  mentioned above. Finiteness and continuity, 
  both jointly
  and separately, are almost entirely 
  independent of the causal hierarchy,
  Theorem \ref{thm:independence}. To prove this result
  we use new class of generalised time functions 
  introduced by the authors.
  
  In \cite{RennieWhale2016} the authors' gave the following characterisation of
  the finiteness
  of the Lorentzian distance and proved a version of the
  Lorentzian distance formula, particular versions of which
  were proved in \cite{franco2010global,Minguzzi2017,moretti2003aspects}: 
  see \cite{franco2018lorentzian} for a review and 
  \cite{Canarutto2019} for a recent reformulation.
  \begin{theorem}[Finiteness of the Lorentzian distance {\cite{RennieWhale2016}}]
    \label{theorem_finiteness}
    Let $(M,g)$ be a Lorentzian manifold.
    The Lorentzian distance is finite if and only if there exists
    a function $f:M\to\R$ so that 
    $\esssup g(\grad f,\grad f)\leq -1$.
  \end{theorem}
  Such a function is necesssarily monotonic on timelike curves.
  \begin{theorem}[The Lorentzian distance formula]
    \label{theorem_LDis}
    If $(M,g)$ has finite Lorentzian distance then for all $p,q\in M$
    \[
      d(p,q)=\inf\bigl\{\abs{f(q)-f(p)}:\  f:\,M\to\R, \ 
        \esssup g(\grad f, \grad f)\leq -1\bigr\}.
    \]
  \end{theorem}
  Theorems \ref{theorem_finiteness} and \ref{theorem_LDis}
  were proven by showing how to construct
  a sufficiently large number of functions,
  that we call {\bf \suf s}, with the appropriate properties,
  see Definition \ref{def.functioninducedbyachronalboundary}.
  It is natural to wonder if these functions also characterise continuity
  and could provide a converse to the globally hyperbolic result
  \cite[Lemma 4.5]{BeemEhrlichEasley1996} mentioned in the
  first paragraph.

  The versions of the Lorentzian distance formula 
  reviewed in 
  \cite{franco2018lorentzian}, all require the assumption of a
  condition in the causal hierarchy, e.g. global hyperbolicity or stable
  causality. This reduces their applicability,
  but allows for stronger regularity
  conditions on the functions
  involved.
  In particular, Minguzzi \cite[Theorem 97]{Minguzzi2017} has shown that
  in stably causal manifolds
  the finiteness and continuity of the Lorentzian distance
  is equivalent to the Lorentzian distance formula holding.
  
  We have already shown that if the Lorentzian distance
  is continuous then all surface functions are 
  continuous \cite[Corollary 3.15]{RennieWhale2016}.
  With the assumption that the Lorentzian distance is finite
  we prove the converse to 
  \cite[Corollary 3.15]{RennieWhale2016}
  in Theorem \ref{theorem.fcdic}.
  \begin{restatable*}{theorem}{finitecontinuoussurface}\label{theorem.fcdic}
    Let $M$ have finite Lorentzian distance.
    The Lorentzian distance is continuous if and only if
    every function in $\eikonal$ is continuous.
  \end{restatable*}

  Thus Theorems \ref{theorem.fcdic} and \ref{theorem_LDis}
  provide
  a generalisation of Minguzzi's result that drops the assumption
  of stably causality. 
  \begin{corollary}\label{corollary_finite_continuous_existence_and}
    Let $(M,g)$ be a Lorentzian distance.
    The Lorentzian distance is finite
    and continuous if and only if
    $\eikonal\neq\EmptySet$ and
    every element of $\eikonal$ is continuous.
  \end{corollary}

  In Theorem \ref{thm.lddis} we rephrase the necessary and sufficient
  conditions for continuity of the Lorentzian distance in terms of the limiting
  behaviour of lengths of curves. In the process we remove the requirement for
  the Lorentzian distance to be finite, and this relies heavily on 
  the characterisation of continuity by the functions $\eikonal$.
  \begin{restatable*}{theorem}{messyContinuiutyEquiv}\label{thm.lddis}
    Let $(M,g)$ be a Lorentzian manifold. The Lorentzian distance is
    discontinuous if and only if there exists $x,y\in M$, with
    $d(x,y)<\infty$, and
    $(x_i)_{i\in\N}\subset M$ a future directed sequence converging to
    $x$, $(y_i)_{i\in\N}\subset M$ a past directed sequence converging
    to $y$ and a sequence of curves $(\gamma_i)_{i\in\N}$ so that
    for all $i\in\N$,
    $\gamma_i\in\Omega_{x_i,y_i}$ and at least one of the following
    is true:
    \begin{itemize}
      \item $\lim_{i\to\infty}L\left(\gamma_i\setminus\future{x}\right)>0$, 
      \item $\lim_{i\to\infty}L\left(\gamma_i\setminus\past{y}\right)>0$.
    \end{itemize}
  \end{restatable*}
%
  Section \ref{sec.conformal} studies the relationship between
  the conformal structure and the finiteness and continuity of the Lorentzian
  distance, as in \cite{Minguzzi2009}. We show that 
  the causes of discontinuity of the Lorentzian distance are
  of two kinds: those invariant under conformal transformations,
  and those which can be introduced or removed by conformal transformations.
  We follow up in subsection 
  \ref{section_conformal_applications} by presenting 
  a few results, both new and old, 
  that follow easily from the preceding material.
  These results
  relate the Lorentzian distance to causal structure.
%
%
  The results of Sections \ref{characterising} and
  \ref{sec.conformal} taken together
  indicate that there
  is a weak relationship between finiteness and continuity of the
  Lorentzian distance and the causal hierarchy.
  This confirms the impression left by \cite{Minguzzi2009}.

  In Section \ref{sec:fin-cts-caus} we show that this weak relationship
  is very weak. 
  We give a series of examples showing that
  finiteness and
  continuity of the Lorentzian distance, both jointly and
  separately, are almost entirely independent
  of all the ``standard'' causality conditions weaker than global
  hyperbolicity, see Theorem
  \ref{thm:independence}.
  Causal structure, the Lorentzian distance function 
  and the \suf s that we employ are all global objects, 
  of which the causal structure is always conformally invariant. 
  Our results indicate that while the causal structure and 
  good properties of the Lorentzian distance are 
  largely independent, properties of the Lorentzian distance 
  and the \suf s are tightly intertwined, c.f. Theorems \ref{theorem_finiteness},
  \ref{theorem_LDis} and \ref{theorem.fcdic}.

  \subsection{Notation}

    A Lorentzian manifold, $(M,g)$, is a smooth, 
    Hausdorff, paracompact manifold, $M$,
    equipped with a Lorentzian metric, $g$. We will not always mention the metric
    explicitly. Two manifolds, $(M,g), (N,h)$ are conformally related
    if $M=N$ and there exists $\Omega:M\to\R$ so that $h=\Omega^2g$.
    
    When necessary we will explicitly mention the metric, for example
    $L(\gamma;g)$ is the arc length of $\gamma$ with respect to $g$,
    $\future{x;g}$ is the future of $x$ with respect to $g$, 
    $d(x,y;g)$ is the Lorentzian distance between $x$ and $y$ with
    respect to $g$, and so on.
    Unless otherwise mentioned,
    we assume that curves are piecewise $C^1$, with everywhere non-zero tangent
    vector, and we treat them as both sets
    and functions.
    For example,
    $\gamma:[0,1]\to M$ is a curve and if $x\in M$ then by 
    $\gamma\cap\future{x}$ we mean the subcurve of $\gamma$ whose
    image is $\gamma([0,1])\cap\future{x}$.
    We make use of several
    of limit curve results for continuous causal curves.
    These results 
    are scattered in a variety of sources, we have collected
    those that we need in Appendix \ref{appendix:limitcurves}.
    
    A subset, $B$, of $M$ is an achronal boundary
    if there exists $F\subset M$ so that $F=\future{F}$ and $B=\partial F$.
    A set $U\subset M$ is convex is any two points in $U$ can be joined by a
    unique geodesic curve contained in $U$.
    
    Given $x,y\in M$ let $\paths{x}{y}$ be the, possibly
    empty, set of future directed piecewise
    smooth timelike
    curves from $x$ to $y$. 
    \begin{definition}\label{def.functioninducedbyachronalboundary}
      Let $S\subset M$ be an achronal boundary such that
      $M=\future{S}\cup S\cup \past{S}$ and
      for all $x\in M$, $d(x,S)$ and $d(S,x)$ are finite.
      The function 
      $\tau_S:M\to\R$ defined by
      \[
        \tau_S(x):=\left\{\begin{aligned}
            &d(S, x) && x\in\future{S}\\
            &0 && x\in S\\
            -&d(x,S) && x\in\past{S}
          \end{aligned}\right.,
      \]
      will be called the \suf \ of $S$.
      The set of all \suf s induced
      by an achronal surface as above
      will be denoted
      $\eikonal$. 
    \end{definition}
    
    The \suf s are differentiable a.e., monotonically increasing on all
    timelike curves, and satisfy $\esssup g(\nabla \tau_S,\nabla\tau_S)\leq
    -1$: see 
    \cite{RennieWhale2016}. If the Lorentzian distance is finite
    then at least one of these functions exists, \cite{RennieWhale2016}.

    Otherwise our notation and definitions follow \cite{BeemEhrlichEasley1996}.

\section{Characterising continuity}
\label{characterising}
  We already know from the proof of 
  Theorem \ref{theorem_finiteness}, 
  see \cite{RennieWhale2016}, that finiteness of the
  Lorentzian distance is equivalent to the existence of a 
  \suf.
  In this section we give two characterisations of 
  continuity of the 
  Lorentzian distance: one in terms of \suf s, and one in terms of 
  the behaviour of lengths of curves. Our first lemma is ``half''
  of our final result on the behaviour of lengths of curves.

  \begin{lemma}\label{lem.basicassumption}
    Let $(M,g)$ be a Lorentzian manifold, and
    $(x,y)\in M\times M$. If the Lorentzian distance is discontinuous
    at $(x,y)$ then for all future directed sequences $(x_i)_{i\in\N}$
    converging to $x$, and
    all past directed sequences $(y_i)_{i\in\N}$ converging to
    $y$, 
    there exists a sequence of future directed curves
    $(\gamma_i)_{i\in\N}$ so that $\gamma_i\in\paths{x_i}{y_i}$,
    $\lim_{i\to\infty}\len{\gamma_i}=\lim_{i\to\infty}d(x_i,y_i)>d(x,y)$
    and at least one of the following
    is true:
    \begin{itemize}
      \item $\lim_{i\to\infty}L\left(\gamma_i\setminus\future{x}\right)>0$, 
      \item $\lim_{i\to\infty}L\left(\gamma_i\setminus\past{y}\right)>0$.
    \end{itemize}
  \end{lemma}
  \begin{proof}
    Since the Lorentzian distance is discontinuous at $(x,y)$ there exists
    a sequence, $((u_i, v_i))_{i\in\N} \subset M\times M$, so that
    $(u_i,v_i)\to(x,y)$ and $\lim_{i\to\infty}d(u_i, v_i)\neq d(x,y)$.
    Since the Lorentzian distance is lower semi-continuous,
    \cite[Lemma 4.4]{BeemEhrlichEasley1996},
    $\lim_{i\to\infty}d(u_i, v_i) > d(x,y)$.

    Let $(x_i)_{i\in\N}$ be a future directed sequence converging to $x$ and
    $(y_i)_{i\in\N}$ a past directed sequence converging to $y$.
    For all $i\in\N$ there exists $N\in\N$ so that 
    for all $j\geq N$,
    $u_j,v_j\in\past{y_i}\cap\future{x_i}$,
    hence $d(x_i,y_i)\geq \lim_{j\to\infty}d(u_j,v_j)$. 
    Thus 
    $\lim_{i\to\infty}d(x_i,y_i)\geq\lim_{i\to\infty}d(u_i,v_i)>d(x,y)$.
    By construction $\Omega_{x_i,y_i}\neq\EmptySet$.
    The existence of a sequence of curves
    $(\gamma_i)_{i\in\N}$ so that $\gamma_i\in\paths{x_i}{y_i}$,
    $\lim_{i\to\infty}\len{\gamma_i}=\lim_{i\to\infty}d(x_i,y_i)>d(x,y)$
    now follows directly from the definition of the Lorentzian distance.

    It remains to show that
    at least one of 
      $$
      \lim_{i\to\infty}L\left(\gamma_i\setminus\future{x}\right)>0,\quad\mbox{or}\quad
       \lim_{i\to\infty}L\left(\gamma_i\setminus\past{y}\right)>0,
       $$
    is true.
    For each $i\in\N$, let $\gamma_i^-=\gamma_i\setminus\future{x}$ and 
    $\gamma_i^+=\gamma_i\setminus\past{y}$.
    These sub-curves are not necessarily disjoint.
    We have that $L(\gamma_i) = L(\gamma_i^-\cup\gamma_i^+) + 
      L(\gamma_i\cap\future{x}\cap\past{y})$.
    Since $\lim_{i\to\infty}L(\gamma_i\cap\future{x}\cap\past{y})\leq d(x,y)$
    and
    $\lim_{i\to\infty}L(\gamma_i)> d(x,y)$ it is the case that
    $\lim_{i\to\infty}L(\gamma_i^-\cup\gamma_i^+)>0$.
    This implies that at least one of
    $\lim_{i\to\infty}L(\gamma_i^-)>0$ or
    $\lim_{i\to\infty}L(\gamma_i^+)>0$ is true as required.
  \end{proof}

  The ``obvious'' converse of Lemma \ref{lem.basicassumption} is not true. 
  Namely we can have a sequence of curves with the 
  strange `limiting length' behaviour, but still have continuity 
  of the Lorentzian distance at the point in question. The following 
  example illustrates this behaviour.

  \begin{example}
  \label{example_missing}
    Let $M=\R^2\setminus\{(1,t)\in\R^2:t\in[1,2]\}$. 
    Let $V=\{(1 + s, t + s)\in\R^2: t\in(1,2), s\in(0,\infty)\}$.
    Let $x=(0,0)$, $y=(0,4)$,
    $(x_i)_{i\in\N}$ be a future directed sequence 
    converging to $x$ and $(y_i)_{i\in\N}$ be a past directed sequence
    converging to $y$.
    We now show that $d$ is continuous at $(x,y)$
    and that there exists a sequence of
    curves $(\gamma_i)_{i\in\N}$ so that for all 
    $i\in\N$, $\gamma_i\in\Omega_{x_i,y_i}$ and
    $\lim_{i\to\infty}L(\gamma_i\setminus\future{x})>0$.
    Hence we give a counter example to the converse of Lemma
    \ref{lem.basicassumption}. The situation is depicted in Figure \ref{fig-1}.

    We show that any curve from $x_i$ to $y_i$ through $V$ must have
    length less than $4$. 
    As the metric is flat, for any sequence of curves $(\gamma_i)_{i\in\N}$
    so that for each $i$, $\gamma_i\in\paths{x_i}{y_i}$ we have that
    $\lim_{i\to\infty}L(\gamma_i\setminus\past{y})=0$ and
    $\lim_{i\to\infty}L(\gamma_i\setminus\left(\future{x}\cup V\right))=0$.
    Hence to get an upper bound for
    $\lim_{i\to\infty}L(\gamma_i)$ we can calculate the limit of the lengths
    of the longest timelike geodesic from $y_i$ to the boundary of $V$ and
    the length of the longest timelike geodesic in $V$.
    The limit of the lengths of the longest timelike geodesic from $y_i$
    to the boundary of $V$ will be equal to the length of longest timelike geodesic
    from $y$ to $V$.

    The portion of $V$ in the past of $y$ is the quadrilateral whose four
    vertices are $(1,1)$, $(1,2)$, $\left(\frac{3}{2},\frac{5}{2}\right)$, 
    and $(2,2)$.
    The reverse triangle inequality implies that the longest timelike
    geodesic from $y$ to the line segment from $(1,2)$ to 
    $\left(\frac{3}{2},\frac{5}{2}\right)$ is the straight line from
    $y$ to $(1,2)$. Similarly the reverse triangle inequality implies
    that the longest timelike geodesic in $V$ is the straight line
    from $(1,1)$ to 
    $\left(\frac{3}{2},\frac{5}{2}\right)$. A little care is needed here
    since the points $(1,2)$ and $(1,1)$ are not in the manifold.
    The lengths of the two geodesics can be calculated as
    $\sqrt{3}$ and $\sqrt{2}$.
    Since $\sqrt{3}+\sqrt{2}<4$ we have demonstrated the claim.

    Since the metric
    is flat, geodesics are straight lines, thus
    a simple calculation shows that $d(x,y)=4$. This implies that
    $d$ is continuous at $(x,y)$.
    It remains to show that
    $\lim_{i\to\infty}L(\gamma_i\setminus\future{x})>0$. This
    follows immediately as $V\subset M\setminus\future{x}$ and
    $V$ is open.
    Note that as $V$ is open and as $V\subset\future{x_i}$ for all $i\in\N$
    then if $v\in V$ then $d$ is discontinuous at $(x,v)$.
  \end{example}

  \begin{figure}
    \begin{center}
      \hspace*{\stretch{1}}
      \begin{tikzpicture}
        \draw[dashed] (-2.5, 2.5) -- (0,0) -- (0.97,1);
        \filldraw (0,0) circle (2pt); 
        \node at (-1.1,0) {$(0,0)=x$};
        \filldraw (0,-1) circle (2pt) node[anchor=east] {$x_i$};
        \draw[dashed] (0.97,2.5) -- (3,4.5);

        \fill[pattern=north west lines] (0.97, 1) -- (3, 3) -- (3, 4.5) -- (0.97, 2.5);
        \draw[fill=black, rounded corners=0.03cm] (0.97,1) rectangle (1.03,2.5);

        \filldraw (0,5) circle (2pt); 
        \node at (-1.1, 5) {$(0,4)=y$};
        \filldraw (0,6) circle (2pt) node[anchor=east] {$y_i$};
        \draw[dashed] (-2.5, 2.5) -- (0,5) -- (2.5, 2.5);

        \draw (0,-1) ..controls(2.2,2.5).. (0,6) node[anchor=west, pos=0.15] {$\gamma_i$};

        \node(V) at (2, 5.5) {$V$};
        \draw[-{Latex[length=3mm]}] (V) -- (2.5,3);

      \end{tikzpicture}
      \hspace*{\stretch{1}}
      \begin{tikzpicture}
        \draw[dashed] (-2.5, 2.5) -- (0,0) -- (0.97,1);
        \filldraw (0,0) circle (2pt) node[anchor=east] {$x$};
        \draw[dashed] (0.97,2.5) -- (3,4.5);
        \fill[pattern=north west lines] (0.97, 1) -- (3, 3) -- (3, 4.5) -- (0.97, 2.5);
        \draw[fill=black, rounded corners=0.03cm] (0.97,1) rectangle (1.03,2.5);

        \filldraw (0,5) circle (2pt); 
        \node at (-1.1, 5) {$(0,4)=y$};
        \draw[dashed] (-2.5, 2.5) -- (0,5) -- (2.5, 2.5);

        \draw (0,5) -- (1,2.5);
        \draw (1,1) -- (1.75,3.25);

        \filldraw[fill=white, white] (0,-1) circle (2pt) node[anchor=east] {$x_i$};

        \filldraw (1,2.5) circle (2pt) node[anchor=east] {$(1,2)$};
        \filldraw (1,1) circle (2pt) node[anchor=east] {$(1,1)$};
        \filldraw (2.5,2.5) circle (2pt) node[anchor=west] {$(2,2)$};
        \node(I) at (1.75,3.25) {};
        \filldraw (I) circle (2pt);
        \node(L) at (2,5.5) {$(1.5,2.5)$};
        \draw[-{Latex[length=3mm]}] (L) -- (I);
      \end{tikzpicture}
      \hspace*{\stretch{1}}
    \end{center}
    \caption{
    An illustration of the proof that the
    Lorentzian distance is continuous at $(x,y)=((0,0),(0,4))$
    given in Example \ref{example_missing}.
    In each diagram the 
    boundary of the future $x$ is given by the upward sloped dashed lines
    and the boundary of the past of $y$ is given by the downward sloped
    dashed lines.
    The hashed area
    is $V$ and the black stripe 
    is 
    the line
    which has been removed from the manifold. 
    In the left hand diagram $x_i$ is a representative element of the
    future directed sequence converging to $x$, $y_i$ is a representative
    of the past directed sequence converging to $y$ and
    $\gamma_i$ is a representation of the sequence of timelike curves
    given in the example.
    In the right hand diagram
    the two thinner lines from
    $y$ to $V$ and inside $V$ are the
    two maximal timelike geodesics that are used to show that the maximum
    length of a curve from $(0,0)$ to $(0,4)$ is less that $\sqrt{2}+\sqrt{3}$.}
    \label{fig-1}
  \end{figure}
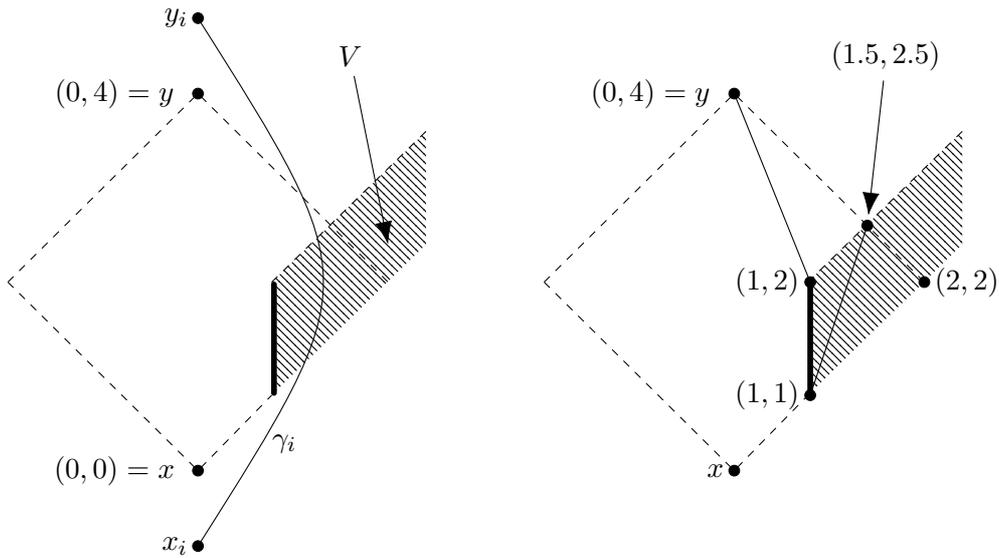

  It is the case, however, that the presence of a sequence of curves
  $(\gamma_i)_{i\in\N}$ with the properties described in the conclusion
  of Lemma \ref{lem.basicassumption} does imply discontinuity of the
  Lorentzian distance. We just don't know \emph{where} the discontinuity
  occurs. Thus if insistence that the discontinuity occurs at
  $(x,y)$ is dropped then this form of
  the converse of Lemma \ref{lem.basicassumption}
  does hold. The statement is in Theorem 
  \ref{thm.lddis}.
  The rest of this section sets out to prove this.
  
  The proof of the converse relies on \suf s
  to detect when discontinuity occurs,
  without needing to know precisely 
  where the discontinuity is located. Hence,
  along the way to our ultimate goal, 
  we prove a simple characterisation of continuity 
  in Lorentzian manifolds with finite Lorentzian distance.

  The following lemma is a technical result which summarises a technique
  first used in the proof of the Lorentzian
  distance formula given in \cite{RennieWhale2016}.

  \begin{lemma}\label{lem.boundaryconstructionwithtwopoints}
    Let $M$ have finite Lorentzian distance.
    If $x\in M$ and $y\in\future{x}$ then there exists
    an achronal boundary $S$ so that 
    \begin{enumerate}
      \item $M=\past{S}\cup S\cup\future{S}$,
      \item $x\in S$,
      \item if $\gamma$ is a timelike curve from $S$ to $y$
        then $\gamma\cap S\subset\partial\future{x}\cap S$, and,
      \item for all $z\in M$, $d(z,S)<\infty$ and $d(S,z)<\infty$.
    \end{enumerate}
  \end{lemma}
  \begin{proof}
    Since the Lorentzian distance is finite \cite[Lemma 3.7]{RennieWhale2016}
    and 
    \cite[Lemma 3.12]{RennieWhale2016} imply 
    that there exists $S_1\subset M$
    an achronal boundary so that 
    $M=\future{S_1}\cup S_1\cup\past{S_1}$ and that, for all $z\in M$
    $d(z,S_1)<\infty$ and $d(S_1, z)<\infty$. We now modify $S_1$ by
    adding/removing bits to its past/future to define a surface
    $S$ with the required property. The situation is depicted in 
    Figure \ref{fig-2}.

    Let $S_2=\partial\left(\past{y}\cup\past{S_1}\right)$. 
    Since $\past{y}\cup\past{S_1}$ is a past set $S_2$ is an achronal boundary.
    If $y\in\past{S_1}\cup S_1$ then $S_1=S_2$ so that
    $M=\future{S_2}\cup S_2\cup\past{S_2}$.
    If $y\in\future{S_1}$ then $y\in S_2$
    and hence $\past{S_2}=\past{y}\cup\past{S_1}$.
    Some definition chasing shows that
    $\future{S_2}=\future{S_1}\setminus\closure{\past{y}}$.
    Since
    $S_1=S_1\setminus\past{y}\cup \bigl(S_1\cap\past{y}\bigr)$
    we have that
    $M=\future{S_2}\cup S_2\cup\past{S_2}$ as required.
    
    By construction for all $z\in M$, 
    $d(z, S_2)\leq \max\{d(z, S_1), d(z, y)\}$. Since the Lorentzian
    distance is finite $d(z,S_2)<\infty$.
    Similarly since $\future{S_2}\subset\future{S_1}$,
    for all $z\in M$, $d(S_2, z)\leq d(S_1, z)<\infty$.
    
    Let $S=\partial\left(\future{x}\cup\future{S_2}\right)$. The time dual
    of the argument about $S_2$ shows that
    $S$ is achronal,
    $M=\future{S}\cup S\cup\past{S}$ and that
    for all $z\in M$, $d(z,S)$ and $d(S, z)$ are finite valued.

    We now show that $x\in S$. If $y\in\future{S_1}\cup S_1$ then
    $y\in S_2$ so that $x\in\past{S_2}$. If $y\in\past{S_1}$
    then $y\in\past{S_2}$ so that $x\in\past{S_2}$.
    Thus, in either case $x\in\past{S_2}$.
    By construction this implies that $x\in S$
    as required.

    Let $\gamma:[0,1]\to M$ be a future directed curve from
    $S$ to $y$. 
    By construction 
    $S=S_2\setminus\future{x}\cap\partial\left(\future{x}\cap\past{S_2}\right)$.
    Since $y\in\future{x}$, the point
    $\gamma(0)$ is contained in 
    $\partial\left(\future{x}\cap\past{S_2}\right)=
      \partial\left(\future{x}\cap S\right)$.
    As $\gamma\cap S=\{\gamma(0)\}$ we have the result.
  \end{proof}

  \begin{corollary}\label{cor.buildachronalboundarysuitablefordiscontinuousstudy}
    Let $M$ have finite Lorentzian distance,
    $(x,y)\in M\times M$, $(x_i)_{i\in\N}$ be a future directed
    sequence converging to $x$, $(y_i)_{i\in\N}$ a past directed sequence
    converging to $y$ and $(\gamma_i)_{i\in\N}$ a sequence of curves
    so that for all $i\in\N$, $\gamma_i\in\Omega_{x_i, y_i}$.
    Then there exists an achronal boundary $S\subset M$ so that
    $M=\future{S}\cup S \cup\past{S}$, 
    $x\in S$ and there exists $N\in\N$ so that for
    all $i\geq N$,
    $\gamma_i\setminus\future{x}\subset\past{S}$.
  \end{corollary}
  \begin{proof}
    Choose $n\in\N$ and apply Lemma \ref{lem.boundaryconstructionwithtwopoints}
    using $x\in M$ and $y_n\in\future{x}$.
    The result is an achronal boundary $S$ so that
    \begin{enumerate}
      \item $M=\past{S}\cup S\cup\future{S}$,
      \item $x\in S$,
      \item if $\gamma$ is a timelike curve from $S$ to $y_n$
        then $\gamma\cap S\subset S\cap \partial\future{x}$, and,
      \item for all $z\in M$, $d(z,S)<\infty$ and $d(S,z)<\infty$.
    \end{enumerate}
    The situation is depicted in Figure \ref{fig-3}.
    We now show that 
    for all $i\geq n$, $\gamma_i\setminus\future{x}\subset\past{S}$.
    Let $i\geq n$. Since $x_i\in\past{x}\subset\past{S}$ and
    $y_i\in\future{x}\subset\future{S}$, \cite[Propsotion 3.15]{Penrose1972}
    implies that there exists a unique point $z\in\gamma_i\cap S$.
    As $\gamma_i\cap\future{S}$ is a timelike curve from $S$ to $y$
    we know that $z\in S\cap \partial\future{x}$.
    Hence $\gamma_i\setminus\future{x}\subset\past{S}$ as required.
  \end{proof}

  \begin{figure}[p]
    \begin{center}
      \hspace*{\stretch{1}}
      \begin{tikzpicture}
        \draw[dashed] (-1.5,-0.5) -- (0, 1) node[anchor=south]{$y$} -- (1.5, -0.5);
        \draw (-2,0) node[anchor=north]{$S_1$} -- (2, 0);
        \draw[dashed] (-1,0.5) -- (0.5, -1) node[anchor=north]{$x$} -- (2, 0.5);
      \end{tikzpicture}
      \hspace*{\stretch{1}}
      \begin{tikzpicture}
        \draw[dashed] (-1.5,-0.5) -- (-1, 0);
        \draw[dashed] (1,0) -- (1.5, -0.5);
        \draw (-2,0) node[anchor=north]{$S_2$} -- (-1, 0) -- (0,1) node[anchor=south]{$y$} -- (1,0) -- (2,0);
        \draw[dashed] (-1,0.5) -- (0.5, -1) node[anchor=north]{$x$} -- (2, 0.5);
      \end{tikzpicture}
      \hspace*{\stretch{1}}
      \begin{tikzpicture}
        \draw (-2,0) node[anchor=north]{$S$} -- (-1, 0) -- (-0.75, 0.25) -- (0.5, -1) node[anchor=north]{$x$} -- (1.25, -0.25) -- (1.5, 0) -- (2,0);
        \draw[dashed] (-1,0.5) -- (-0.75, 0.25);
        \draw[dashed] (1.5, -0) -- (2, 0.5);
        \draw[dashed] (-0.75,0.25) -- (0, 1) node[anchor=south]{$y$} -- (1.5, -0.5);
        \draw[dashed] (-1.5,-0.5) -- (-1, 0);
      \end{tikzpicture}
      \hspace*{\stretch{1}}
    \end{center}
    \caption{The three figures above illustrate the construction of $S$ given in
    the proof of Lemma \ref{lem.boundaryconstructionwithtwopoints} when 
    $y\in\future{S_1}$.}
    \label{fig-2}
  \end{figure}
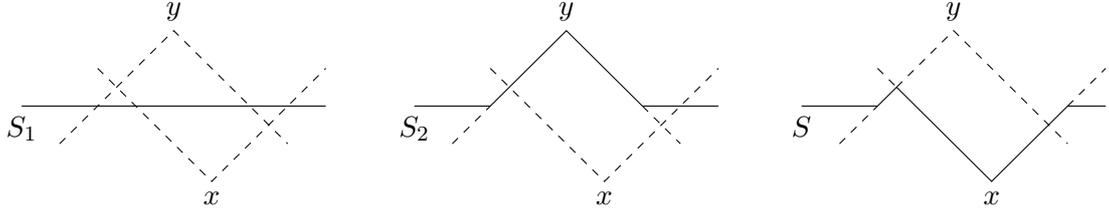

  \begin{figure}[p]
    \begin{center}
      \hspace*{\stretch{1}}
      \begin{tikzpicture}
        \draw (-4,0) node[anchor=north]{$S$} -- (-2, 0) -- (-1.5, 0.5) -- (1, -2) node[anchor=east]{$x$} -- (2.5, -0.5) -- (3, 0) -- (4,0);
        \draw[dashed] (-3,2) -- (-1.5, 0.5);
        \draw[dashed] (3, -0) -- (3.5, 0.5);
        \draw[dashed] (-1.5,0.5) -- (0, 2) node[anchor=south]{$y_n$} -- (3.5, -1.5);
        \draw[dashed] (-3,-2) -- (0, 1) node[anchor=south]{$y_i$} -- (3.5, -2.5);
        \draw[dashed] (-3,-3) -- (0, 0) node[anchor=east]{$y$} -- (3.5, -3.5);
        \draw[dashed] (-3,-1) -- (-2, 0);
        \draw[dashed] (-3, 1) -- (1, -3) node[anchor=north]{$x_i$} -- (3.5, -0.5);

        \filldraw (1, -2) circle (2pt);
        \filldraw (0, 2) circle (2pt);
        \filldraw (0, 1) circle (2pt);
        \filldraw (0, 0) circle (2pt);
        \filldraw (1, -3) circle (2pt);

        \draw[dash pattern=on 40pt off 5pt on 5pt off 5pt on 5pt off 5pt on 5pt off 5pt on 5pt off 5pt on 5pt off 5pt on 5pt off 5pt on 5pt off 5pt on 5pt off 5pt on 5pt off 5pt on 5pt off 5pt on 5pt off 5pt on 5pt off 5pt on 5pt off 5pt on 5pt] (1,-3) .. controls(1.6,-1).. (0,1);
      \end{tikzpicture}
      \hspace*{\stretch{1}}
    \end{center}
    \caption{The figure presents the essence of the geometric argument
    used in the proof of Corollary \ref{cor.buildachronalboundarysuitablefordiscontinuousstudy} to show that
    $\gamma_i\setminus\future{x}\subset\past{S}$. The curve from 
    $x_i$ to $y_i$ represents $\gamma_i$. The solid part of this curve is
    $\gamma_i\setminus\future{x}$ the dashed portion is
    $\gamma_i\cap\future{x}$.
    Note that $d(x_i,S)>L(\gamma_i\setminus\future{x})$.
    Hence if $\lim_{i\to\infty}L(\gamma_i\setminus\future{x})>0$
    then $\lim_{i\to\infty}d(x_i, S)>0$.}
    \label{fig-3}
  \end{figure}
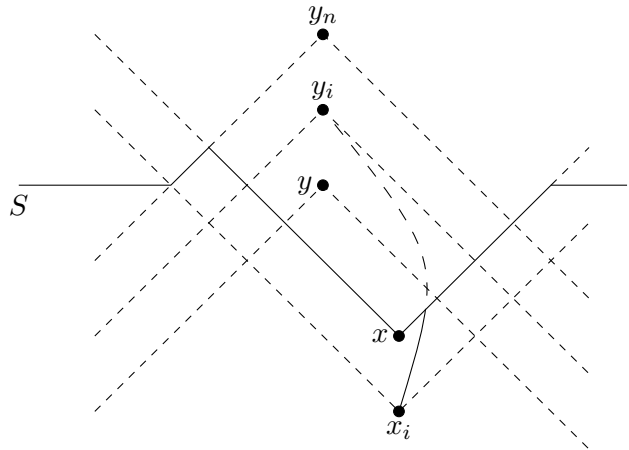

  \begin{corollary}\label{cor.buildaneikonalfunction}
    If in addition to the assumptions of Corollary
    \ref{cor.buildachronalboundarysuitablefordiscontinuousstudy},
    $\lim_{i\to\infty}L(\gamma_i\setminus\future{x})>0$ then
    the \suf\  of $S$ is discontinuous at $x$.
  \end{corollary}
  \begin{proof}
    By assumption the \suf\ of $S$ satisfies
    $\tau_S(x_i)=d(x_i, S)>
      \lim_{i\to\infty}L(\gamma_i\setminus\future{x})>0=\tau_S(x)$.
    Since $(x_i)_i$ converges to $x$, $\tau_S$ is discontinuous at $x$.
  \end{proof}
  
  Observe that Corollary \ref{cor.buildaneikonalfunction}
  does not say that the Lorentzian distance is discontinuous
  at $x$.
  
  \finitecontinuoussurface
  \begin{proof}
    It is already known that if the Lorentzian distance is continuous
    then every element of $\eikonal$ is continuous,
    \cite[Corollary 3.15]{RennieWhale2016}.

    Suppose that the Lorentzian distance is discontinuous.
    Lemma \ref{lem.basicassumption} implies that
    there exists $x,y\in M$, $(x_i)_{i\in\N}\subset M$ a future directed
    sequence converging to $x$, $(y_i)_{i\in\N}\subset M$
    a past directed sequence converging to $y$ and a sequence
    of curves $(\gamma_i)_{i\in\N}$ so that $\gamma_i\in\Omega_{x_i,y_i}$,
    $\lim_{i\to\infty}d(x_i,y_i)=\lim_{i\to\infty}L(\gamma_i)>d(x,y)$ and
    at least one of the following is true:
    \begin{itemize}
      \item $\lim_{i\to\infty}L\left(\gamma_i\setminus\future{x}\right)>0$, 
      \item $\lim_{i\to\infty}L\left(\gamma_i\setminus\past{y}\right)>0$.
    \end{itemize}
    Hence Corollary \ref{cor.buildaneikonalfunction}, or its time dual,
    gives the result.
  \end{proof}

  It is now possible to remove the assumption of finiteness 
  from Theorem \ref{theorem.fcdic} and thus prove
  a converse to Lemma \ref{lem.basicassumption}.

  \messyContinuiutyEquiv

  This theorem expresses the idea that the Lorentzian distance is
  discontinuous if and only if there is a sequence of curves
  whose lengths limit to a non-zero value when the limit ``should''
  be zero (with some regularity assumptions included). 
  To put that into context, the curves $\gamma_i\setminus\future{x}$
  limit to a causal curve in a null surface and 
  so lengths of the curves ``should'' also
  limit to zero. This is what happens in globally hyperbolic manifolds (just 
  as the theorem implies) and
  what happens when the limit curve between $x$ and $y$ exists.
  One way to interpret the theorem is that if the limiting null surface
  isn't really null or when the limit isn't sufficiently uniform
  then discontinuities of the Lorentzian distance result,
  see Examples \ref{example_missing}, \ref{example.infinite},
  and \ref{example.weaksingularity}.

  The theorem does not claim that the Lorentzian distance is
  discontinuous at $(x,y)$. 
  Example \ref{example_missing} shows that
  the conditions of the theorem can be satisfied but the
  Lorentzian distance is continuous at $(x,y)$. What is important
  is that the non-zero limit of the 
  lengths of the given sub-curves implies that
  there exists {\em some pair} $(u,v)\in M\times M$ at which the Lorentzian
  distance is discontinuous. 
  The proof relies, in a non-trival way, on 
  Theorem \ref{theorem.fcdic} to avoid direct specification of
  $(u,v)$.
  That is, by appealing to the continuity of the globally defined
  \suf s the need to 
  explicitly identify
  $u$ and $v$ can be avoided.

  \begin{proof}
    Lemma \ref{lem.basicassumption} proves the ``if'' portion of the result.
    So suppose that there exists $x,y\in M$, with $d(x,y)<\infty$, and
    $(x_i)_{i\in\N}\subset M$ a future directed sequence converging to
    $x$, $(y_i)_{i\in\N}\subset M$ a past directed sequence converging
    to $y$ and a sequence of curves $(\gamma_i)_{i\in\N}$ so that
    $\gamma_i\in\Omega_{x_i,y_i}$ and such that at least one of the
    following is true;
    \begin{itemize}
      \item $\lim_{i\to\infty}L(\gamma_i\setminus\future{x})>0$,
      \item $\lim_{i\to\infty}L(\gamma_i\setminus\past{y})>0$.
    \end{itemize}
    Since these two conditions are time duals of each other
    we can, without loss of generality, assume that for all $i\in\N$,
    $\lim_{i\to\infty}L(\gamma_i\setminus\future{x})>0$.

    If for all $i$, $d(x_i, y_i)=\infty$ then
    as $d(x,y)<\infty$ the Lorentzian distance is discontinuous
    at $(x,y)$.
    So, suppose that there exists $j\in\N$ so that 
    $d(x_j, y_j)<\infty$.

    Let $N=\future{x_j}\cap\past{y_j}$. We consider $N$ as a submanifold
    with the induced metric. By assumption $N$ has a finite Lorentzian distance.
    Let $d_N$ denote the Lorentzian distance on $N$ induced by the
    ambient metric. 
    By definition for any $u,v\in N$, $d_N(u,v)\leq d(u,v)$.
    Since $N$ is a causal diamond for any $u,v\in N$ if $\gamma$ is a timelike
    curve from $u$ to $v$ in $M$ then $\gamma\subset N$. This implies that
    $d_N(u,v)=d(u,v)$.

    For all $i\in\N$, let $\tilde{x}_i=x_{i+j}$, $\tilde{y}_i=y_{i+j}$ and
    $\tilde{\gamma}_i=\gamma_{i + j}$. Then $(\tilde{x}_i)_{i\in\N}\subset N$ 
    is a future directed
    sequence converging to $x$ in $N$, $(\tilde{y}_i)_{i\in\N}\subset N$
    is a past directed sequence converging to $y$ in $N$ and
    $(\tilde{\gamma}_i)_{i\in\N}$ is a sequence of curves in $N$ so that
    $\tilde{\gamma}_i\in\Omega_{\tilde{x}_i,\tilde{y}_i}$.
    By assumption, 
    $\lim_{i\to\infty}L(\tilde{\gamma}_i\setminus\past{x})>0$.
    Corollary \ref{cor.buildaneikonalfunction} implies that there
    exists a discontinuous eikonal function on $N$ induced by
    some achronal boundary.
    Theorem \ref{theorem.fcdic} implies that 
    there exists $u,v\in N$ so that $d_N$ is discontinuous at $(u,v)$.
    Since $N$ is open and $d_N=d|_{N\times N}$, $d$ is discontinuous
    at $(u,v)$.
  \end{proof}

  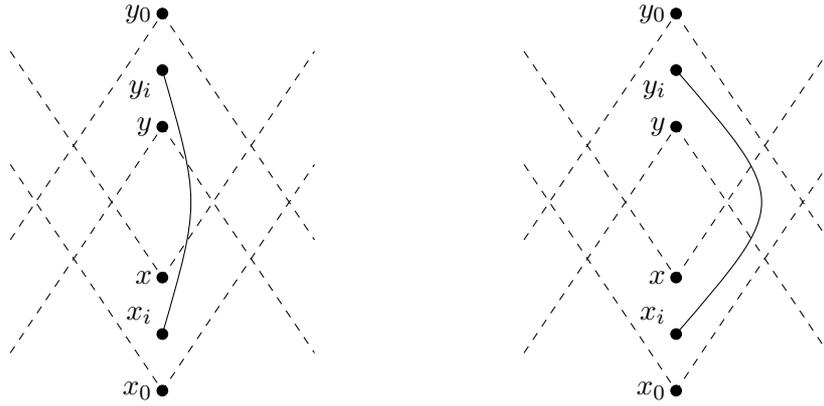
\begin{figure}
    \begin{center}
      \hspace*{\stretch{1}}
      \begin{tikzpicture}
        \draw[dashed] (-2,-2) -- (0, 1) -- (2,-2);
        \draw[dashed] (-2,-0.5) -- (0, 2.5) -- (2,-0.5);
        \draw[dashed] (-2, 2) -- (0,-1) -- (2, 2);
        \draw[dashed] (-2, 0.5) -- (0,-2.5) -- (2, 0.5);
        \filldraw (0, 2.5) circle (2pt) node[anchor=east, align=left] {$y_0$};
        \filldraw (0, 1) circle (2pt) node[anchor=east, align=left] {$y$};
        \filldraw (0,-1) circle (2pt) node[anchor=east, align=left] {$x$};
        \filldraw (0,-2.5) circle (2pt) node[anchor=east, align=left] {$x_0$};
        \draw (0, 1.75) ..controls (0.5,0).. (0, -1.75);
        \filldraw (0, 1.75) circle (2pt) node[anchor=north east, align=left] {$y_i$};
        \filldraw (0, -1.75) circle (2pt) node[anchor=south east, align=left] {$x_i$};
      \end{tikzpicture}
      \hspace{\stretch{1}}
      \begin{tikzpicture}
        \draw[dashed] (-2,-2) -- (0, 1) -- (2,-2);
        \draw[dashed] (-2,-0.5) -- (0, 2.5) -- (2,-0.5);
        \draw[dashed] (-2, 2) -- (0,-1) -- (2, 2);
        \draw[dashed] (-2, 0.5) -- (0,-2.5) -- (2, 0.5);
        \filldraw (0, 2.5) circle (2pt) node[anchor=east, align=left] {$y_0$};
        \filldraw (0, 1) circle (2pt) node[anchor=east, align=left] {$y$};
        \filldraw (0,-1) circle (2pt) node[anchor=east, align=left] {$x$};
        \filldraw (0,-2.5) circle (2pt) node[anchor=east, align=left] {$x_0$};
        \draw (0, 1.75) ..controls (1.5,0).. (0, -1.75);
        \filldraw (0, 1.75) circle (2pt) node[anchor=north east, align=left] {$y_i$};
        \filldraw (0, -1.75) circle (2pt) node[anchor=south east, align=left] {$x_i$};
      \end{tikzpicture}
      \hspace*{\stretch{1}}
    \end{center}
    \caption{In both diagrams the dashed lines give the timelike future/past
    of the point at the tip of the cone and the solid line is the
    timelike geodesic $\gamma_i$. The behaviour of the subcurves
    $\gamma_i\setminus\future{x}$ and $\gamma_i\setminus\past{y}$ characterise
    continuity of the Lorentzian distance. There are two distinct cases shown
    in the diagrams above. On the left is the case when the two
    subcurves are disjoint. This occurs when $\gamma_i\in\future{x}\cap\past{y}$
    is non-empty. On the right is the other case when the two subcurves
    have intersection equal to 
    $\gamma_i\setminus\left(\future{x}\cup\past{y}\right)$. The two cases
    can be treated at the same time, as long as one is aware that
    $\gamma_i\setminus\future{x}\subset\past{y}$ is only true for the
    left diagram.}
  \end{figure}

\section{The Lorentzian distance and conformal transformations}
\label{sec.conformal}

  Both finiteness and continuity of the Lorentzian distance
  can be altered by conformal transformations, e.g.\
  \cite[Theorems 2.4 and 3.6]{Minguzzi2009}.
  Theorem \ref{thm.lddis} implies that this relationship 
  can be studied by understanding how conformal transformations
  change the lengths of curves.

  The set $V$ of Example \ref{example_missing} is 
  conformally invariant,
  hence the ``causes of discontinuity'' of the Lorentzian distance
  in this case can not be removed by a conformal transformation.
  To support the idea that discontinuities 
  of the Lorentzian distance can either be conformally invariant or removable (or introducable!) by conformal transformation,
  we begin our study with two examples which, in contrast to Example \ref{example_missing},
  illustrate discontinuities which can be removed by conformal transformations.

  \begin{example}\label{example.infinite}
    Let $M=\R^2\setminus\{(0,0)\}$. 
    Let $g$ be the metric given by
    \[
      g=\frac{1}{(x^2 + y^2)^2}\left(-dy^2+dx^2\right).
    \]
    Let $x=(0,-1)$.
    Let $F=\future{(0,0)}$ considered as a subset of $M$. A simple calculation
    shows that for all $y\in\overline{F}$, $d(x,y)=\infty$.
    The Lorentzian distance is continuous on $(M, g)$.
    The manifold $(M, g)$ is causally continuous.

    Let $\tilde{F}=\{(x,y)\in M: y > 0, \lvert 2x\rvert < y\}$. One can think
    of $\tilde{F}$ as the future of $(0,0)$ with respect to the metric 
    $-2dy^2 + dx^2$. Let $\rho:M\to [0,1]$ be a bump function so that
    $\rho|_{\tilde{F}}=1$ and $\rho|_{M\setminus{F}}=0$.
    Let 
    $h=\left(\frac{\rho(x,y)}{(x^2+y^2)^2} + 1-\rho(x,y)\right)(-dy^2 + dx^2)$.

    Let $y\in F$. Then there exists $\tilde{y}\in\tilde{F}\cap\past{y}$.
    Again a simple computation shows that $d(x,y;h)=\infty$.
    Let $y\in\partial F$. Then, by definition of $\rho$, $d(x,y;h)<\infty$.
    Thus the Lorentzian distance of $h$ is discontinuous at $(x,y)$.
    Since $M$ is distinguishing, Theorem
    \ref{prop.causallycontinuouspartialconverse}
    implies that
    $(M, h)$ is causally continuous. 

    In Propositions \ref{prop.discontinuousimpliesnoncompact}
    and \ref{prop.noncompactimpliesdiscontinuous} we show how
    this form of discontinuity is related to conformal transformations.
  \end{example}

  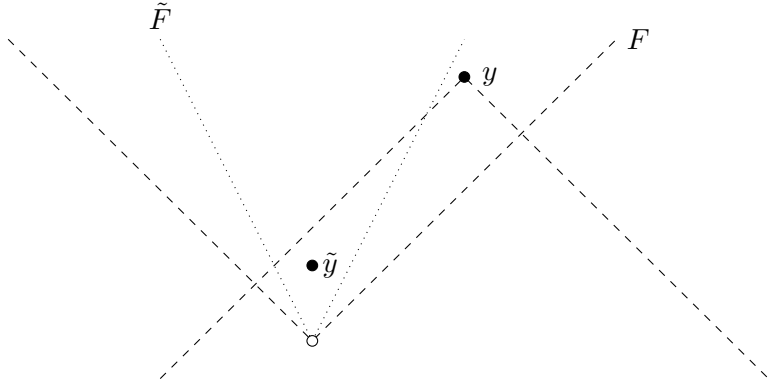
\begin{figure}
    \begin{center}
      \begin{tikzpicture}
        \coordinate (O) at (0,0);
        \draw[dashed] (-4,4) -- (O) -- (4,4);
        \draw[dotted] (-2,4) -- (O) -- (2,4);
        \filldraw[fill=white] (O) circle (2pt);
        
        \coordinate (Y) at (2,3.5);
        \filldraw (Y) circle (2pt);
        \node [anchor=west] at (2.1,3.5) {$y$};
        \draw[dashed] (-2,-0.5) -- (Y) -- (6,-0.5);

        \coordinate (TY) at (0,1);
        \filldraw (TY) circle (2pt);
        \node [anchor=west] at (TY) {$\tilde{y}$};

        \node [anchor=west, align=right] at (4,4) {$F$};
        \node [anchor=south, align=left] at (-2,4) {$\tilde{F}$};
      \end{tikzpicture}
    \end{center}
    \caption{Schematic representation of the construction used in 
    the second part of Example \ref{example.infinite}.}
  \end{figure}

  \begin{example}\label{example.weaksingularity}
    Let $N=\R^2\setminus\{(0,y)\in\R^2:y\in[-1,1]\}$ equipped with the
    metric, $g$, induced by $2$-dimensional Minkowski space.
    Let $M=\R^2\setminus\{(0,0)\}$. We claim that there is a 
    diffeomorphism
    $\phi:N\to M$, so that 
    $\phi(\partial\future{(0,1)})=\partial\future{(0,0)}$
    and $\phi(\partial\past{(0,-1)})=\partial\past{(0,0)}$.
    Equip $M$ with the push forward, $\phi_*g$. 

    Let $x=(-1,-1)$ and $y=(1,1)$. By construction
    $d_{\phi_*g}(x,y)=0$ as there are no timelike curves in $M$ starting at
    $x$ and ending at $y$. Let $(x_i)_{i\in\N}$ be a future directed
    sequence converging to $x$ and for each $i\in\N$ let
    $\gamma_i\in\Omega_{x_i,y}$.
    By construction the length of each $\gamma_i$ in $M$ is the same
    as the length of $\phi^{-1}(\gamma_i)$ and therefore
    is larger than $1$. Thus 
    $\lim_{i\to\infty}d_{\phi_*g}(x_i,y)>d(x,y)$ and $d$ is discontinuous
    at $(x,y)$.

    It remains to show that the claimed diffeomorphism $\phi$
    exists. Let
    \[
      h(t)=\left\{\begin{aligned}
        \exp\left(\frac{-1}{t}\right)&& t > 0\\
        0 && t\leq 0
      \end{aligned}\right.
    \]
    and 
    \[
      k(t)=2\frac{h\left(\frac{1}{2} - t\right)}{h\left(\frac{1}{2} - t\right)
        +h\left(\frac{1}{2}+t\right)} -1.
    \]
    Then
    \[
      \phi(x,y)=\left\{\begin{aligned}
        (x, y - 1) && (x,y)\in\closure{\future{(0,1)}}\\
        \left(x, y + k\left(\frac{y}{\lvert x\rvert+1}\right)\right) && 
          (x,y)\in N\setminus\closure{\future{(0,1)}\cup\past{(0,-1)}}\\
        (x, y + 1) && (x,y)\in\closure{\past{(0,-1)}}
      \end{aligned}\right.
    \]
    is the required diffeomorphism.
    Since $M$ is distinguishing, Theorem
    \ref{prop.causallycontinuouspartialconverse}
    implies that
    $(M, \phi_*g)$ is causally continuous.
  \end{example}
  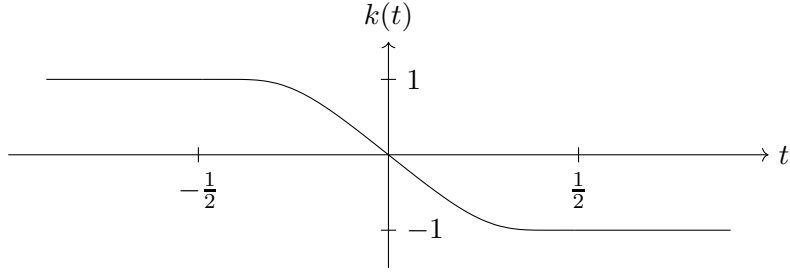
\begin{figure}
    \begin{center}
      \begin{tikzpicture}[xscale=5]
        \draw[->] (0, -1.5) -- (0, 1.5) node[above] {$k(t)$};
        \draw[->] (-1, 0) -- (1, 0) node[right] {$t$};
        \draw[domain=-0.49:0.49,smooth] plot (
          {\x},
          {(2*exp(-1/(0.5-\x))/(exp(-1/(0.5-\x)) + exp(-1/(0.5+\x))) -1}
        );
        \draw[-] (-0.9,1) -- (-0.49,1);
        \draw[-] (0.49,-1) -- (0.9,-1);
        \draw[-] (-0.5,0.1) -- (-0.5,-0.1) node[below,yshift=-0.1] {$-\frac{1}{2}$};
        \draw[-] (0.5,0.1) -- (0.5,-0.1) node[below,yshift=-0.1] {$\frac{1}{2}$};
        \draw[-] (-0.02,1) -- (0.02,1) node[right] {$1$};
        \draw[-] (-0.02,-1) -- (0.02,-1) node[right] {$-1$};

      \end{tikzpicture}
    \end{center}
    \caption{
      The graph of the function $k$ used in Example
      \ref{example.weaksingularity}.
    }
  \end{figure}

  \begin{proposition}\label{prop.discontinuousimpliesnoncompact}
    Let $(M,g)$ be a Lorentzian manifold.
    If the Lorentzian distance is discontinuous at $(x,y)\in M\times M$ then
    there exists a past inextendible timelike curve $\lambda$
    so that for all $u\in\past{x}$, $\lambda\subset\future{u}$
    and $\future{y}\subset\future{\lambda}$.
  \end{proposition}
  \begin{proof}
    By Lemma \ref{lem.basicassumption} there exists
    $(x_i)_{i\in\N}$ a future directed sequence converging
    to $x$, $(y_i)_{i\in\N}$ a past directed sequence 
    converging to $y$, and a sequence of curves
    $(\gamma_i)_{i\in\N}$ so that
    for each $i\in\N$, $\gamma_i\in\Omega_{x_i, y_i}$
    and
    $\lim_{i\to\infty}L(\gamma_i)=\lim_{i\to\infty}d(x_i,y_i)>d(x,y)$.

    Suppose that there exists a compact set $K$ and $N\in\N$ 
    so that for all $i\geq N$, $\gamma_i\subset K$. Then,
    by Lemma \ref{CurLim:Lem.CurveUniformConvergenceInBoundedRegion},
    there exists a future directed continuous causal limit
    curve $\gamma$ from $x$  
    and a
    subsequence of $(\gamma_i)$ that converges to $\gamma$
    uniformly on compact subsets, once each curve is suitably parametrised.
    In an abuse of notation we denote this subsequence by $(\gamma_i)$.
    Lemma \ref{lem.uniformconvergenceimplieslengthconvergence} implies that
    $d(x,y)<\lim_{i\to\infty}L(\gamma_i)\leq L(\gamma)$.
    In particular this implies that $\gamma$ is timelike.
    For each $i\in\N$, $\gamma_i\subset\past{y}$ which implies that
    $\gamma\subset\past{y}$ and hence $L(\gamma)\leq d(x,y)$.
    Since this is a contradiction
    no such $K$ exists. 
    Thus for all compact sets $K\subset M$ and all $N\in\N$ there
    exists $i\geq N$ so that $\gamma_i\not\subset K$.

    Choose $h$ a complete
    Riemannian metric. We are free to parametrise each $\gamma_i$
    with respect to the arc length induced by $h$.
    Hence for each $i$ there exists $b_i>0$ so that
    $\gamma_i:[0,b_i]\to M$.
    From above we know that $b_i\to \infty$.
    Let $\gamma$ be the past directed timelike limit curve
    from $y$ given by
    iterated application of Lemma 
    \ref{CurLim:Lem.CurveUniformConvergenceInBoundedRegion}
    starting with a compact neighbourhood of $y$.
    We will, in an abuse of notation, denote the subsequence of 
    $(\gamma_i)$ that is uniformly convergent to $\gamma$ on
    compact subsets by $(\gamma_i)$. Note that, by construction,
    $\gamma$ is causal.

    Since $b_i\to\infty$ we know that
    $\gamma$ has infinite $h$ arc length.
    This implies that $\gamma$ is inextendible to the past. 
    As $\gamma$ is past inextendible
    $\future{\gamma}$ is a terminal indecomposable future set
    \cite[Theorem 2.3]{Geroch_1972}.
    Theorem 2.1 of \cite{Geroch_1972} implies that there exists
    $\lambda$ a past inextendible timelike curve so that
    $\future{\lambda}=\future{\gamma}$.
    If $\lambda$ had finite $h$ arc length $\lambda$ would be extendible
    to the past as $(M,h)$ is complete. This would imply that
    $\future{\lambda}$ was not a terminal indecomposable future
    set. Since this is a contradiction $\lambda$ must have infinite
    $h$ arc length.

    By construction $\future{y}\subset\future{\lambda}$.
    Let $u\in\past{x}$ then there exists $i\in\N$ so that
    $u\in\past{x_i}$. By construction for all $j>i$,
    $\gamma_j\subset\future{u}$ this implies that 
    $\gamma\subset\closure{\future{u}}$ so that 
    $\lambda\subset\future{\lambda}=\future{\gamma}\subset\future{u}$ as
    required.
  \end{proof}

  \begin{proposition}\label{finite.infintiesimpliesnoncompact}
    Let $(M,g)$ be a chronological Lorentzian manifold.
    If the Lorentzian distance is infinite at $(x,y)\in M\times M$ then
    there exists a past inextendible timelike curve $\lambda$
    so that for all $u\in\past{x}$, $\lambda\subset\future{u}$
    and $\future{y}\subset\future{\lambda}$.
  \end{proposition}
  \begin{proof}
    Let $(\gamma_i)_{i\in\N}\subset\Omega_{x,y}$ be such that
    $L(\gamma_i)\to\infty$. Let $\gamma$ be the past
    directed continuous limit curve
    through $y$. 

    Suppose that $\gamma$ is contained in an open pre-compact subset $U$.
    By Lemma \ref{lem.uniformconvergenceimplieslengthconvergence}
    $L(\gamma)=\infty$.
    In particular,
    as $L(\gamma)>0$ we know that $\gamma$ contains a timelike segment.
    If $\gamma$ contains curve segments which are null then via application
    of \cite[Propositions 2.18 and 2.19]{Penrose1972}, in the
    submanifold $U$,
    we can find a timelike curve with length greater than or
    equal to $\gamma$ contained in $U$.
    In an abuse of notation denote this timelike curve by $\gamma$.
    Note that by construction this new $\gamma$ is also such
    that $L(\gamma)=\infty$.

    Since $\gamma\subset U$,
    $\overline{\gamma}$ is compact.
    Let $\epsilon>0$ and choose an auxiliary Riemannian metric $h$
    on $M$. Since $\overline{\gamma}$ is compact it is covered
    by a finite number of convex normal neighbourhoods of
    $h$-radius less than $\epsilon$. 
    {This implies that
    $\gamma$ must return to at least one of the covering neighbourhoods
    after leaving it.} 
    Since this is true for all $\epsilon$ we can construct a sequence
    $(p_i)_{i\in\N}\subset\gamma$ that has accumulation points in
    $\gamma$. 
    Since $\gamma$ is closed, $U$ is open and $\overline{U}$ is
    compact there exists some subsequence of $(p_i)$ that is convergent.
    In an abuse of notation denote this subsequence by $(p_i)$
    and let $p\in\gamma$ be the limit point.

    Let $q\in\future{p}\cap\gamma$. Since $p\in\past{q}$ and as
    $U$ is open there exists $N\in\N$ so that
    for all $j>N$, $p_j\in\past{q}$. Since the finite $\epsilon$ $h$-radius
    open covers of $\overline{\gamma}$ were made up of convex normal
    neighbourhoods we know that there exists
    some $j> N$ so that $q\in\past{p_j}$.
    This implies that there exists a closed timelike curve.
    The result now follows from the
    same arguments as for the proof of Proposition
    \ref{prop.discontinuousimpliesnoncompact}.
  \end{proof}

  A ``conformal converse'' to
  Proposition \ref{prop.discontinuousimpliesnoncompact} is possible,
  see Proposition \ref{prop.noncompactimpliesdiscontinuous}.
  Minkowski space with a point removed demonstrates that the conformal
  factor is necessary. 

  \begin{lemma}\label{lem.inftydivergence}
    Let $(M,g)$ be a Lorentzian manifold.
    If there exists $x,y\in M$ and a past inextendible timelike curve $\lambda$
    so that
    for all $u\in\past{x}$, $\lambda\subset\future{u}$ and
    $\future{y}\subset\future{\lambda}$,
    then there exists a conformal factor $\Omega$ so that
    for all $z\in\future{y}$, $d(x,z;\Omega^2g)=\infty$.
  \end{lemma}
  \begin{proof}
    Without loss of generality we assume that
    $\lambda$ is parametrised so that for some $b\in\R$,
    $\lambda:[0,b)\to M$.
    Choose
    $f:M\to\R$ a smooth function so that 
    for all $x\in\future{\lambda}$,
    $f(x)\neq 0$ and such that for all $t\in[0,b)$
    \[
      \exp\left(\frac{1}{f(\lambda(t))}\right)\sqrt{-g(\lambda'(t),\lambda'(t))}\geq
      \frac{1}{b-t}.
    \]
    Choose $U\subset M$ an open subset so that $\lambda\subset U$ and
    $\closure{U}\subset\future{\lambda}$.
    Let $\rho:M\to\R$ be such that $\rho|_{\closure{U}}=1$ and
    $\rho|_{M\setminus\future{\lambda}}=0$. 
    Define $\Omega:M\to\R$ by
    $\Omega(x)=\sqrt{1 - \rho(x) + \rho(x)\exp(f(x)^{-1})}$. 
    Let $h=\Omega^2g$ be the Lorentzian metric
    conformally related to $g$ by $\Omega$.

    By construction, for all $a\in\R$,
    \[
      L\left(\lambda; h\right)=
        \int_{a}^b\exp(f(\lambda(t))^{-1})\sqrt{-g(\lambda'(t),\lambda'(t))}dt
        =\infty,
    \]
    where the last equality follows by definition of $f$.

    If $z\in\future{\lambda;g}=\future{\lambda;h}$ 
    then there exists $a\in[0,b)$ so that
    $\lambda([a,b))\subset\past{z;h}$. By assumption this
    implies that for all $u\in\past{x;h}=\past{x;g}$, $d(u,z;h)=\infty$.
  \end{proof}
  
  \begin{proposition}\label{prop.noncompactimpliesdiscontinuous}
    Let $(M,g)$ be a Lorentzian manifold with finite Lorentzian distance.
    If there exists $x,y\in M$ and a past inextendible timelike curve $\lambda$
    so that
    for all $u\in\past{x}$, $\lambda\subset\future{u}$ and
    $\future{y}\subset\future{\lambda}$,
    then there exists a conformal factor $\Omega$ so that
    for all $p\in\past{x}$ and all $q\in\partial\future{y}$
    the Lorentzian distance of
    $(M,\Omega^2g)$ is discontinuous at $(p,q)$.
  \end{proposition}
  \begin{proof}
    Let $\Omega:M\to\R$ be the conformal factor constructed
    in Lemma \ref{lem.inftydivergence}, let $h=\Omega^2g$ and
    let $z\in\partial\future{\lambda;g}=\partial\future{\lambda;h}$.
    Then the definition of $\Omega$ implies that 
    for all $u\in\past{x;h}$,
    $d(u,z;h)=d(u,z;g)$. By assumption $d(u,z;g)<\infty$. Thus
    for all $u\in\past{x;h}$ and $z\in\partial\future{\lambda;h}$
    the Lorentzian distance of $(M,h)$ is discontinuous
    at $(u,z)$, whence the result.
  \end{proof}

  We also have the following lemma, which follows the same proof,
  that can occasionally be useful when the Lorentzian distance is not
  known to be finite. The result is an alternate form of
  \cite[Lemma 2.1]{Minguzzi2009} with a different construction of the 
  required conformal factor.

  \begin{lemma}\label{lem.inftydiscontinuousonachronal}
    Let $(M,g)$ be a Lorentzian manifold.
    If there exists $x,y\in M$ so that
    $d(x,y)=0$ and a past inextendible timelike curve $\lambda$
    so that
    for all $u\in\past{x}$, $\lambda\subset\future{u}$ and
    $\future{y}\subset\future{\lambda}$,
    then there exists a conformal factor $\Omega$ so that
    $(M,\Omega^2g)$ has discontinuous Lorentzian distance at $(x,y)$.
  \end{lemma}
  \begin{proof} Using the function $\Omega$ of 
    Lemma \ref{lem.inftydivergence} shows that for all
    $z\in\future{y}$, $d(x,z;\Omega^2g)=\infty$. The Lorentzian
    distance is therefore discontinuous at
    $(x,y)$ since $d(x,y;\Omega^2g)=0$.
  \end{proof}

  ``Removing'' a discontinuity is harder than introducing one as more control is
  needed over the lengths of curves. We begin our study of finding conformally
  related continuous Lorentzian distances by characterising the conformally
  invariant causal structure that introduces discontinuities of the Lorentzian
  distance.

  \begin{lemma}\label{lem.independenceofshadows}
    Let $(M,g)$ be a Lorentzian manifold and $x\in M$.
    If $(x_i)_{i\in\N}$ and $(u_i)_{i\in\N}$ are future directed
    sequences converging to $x$ then
    $\bigcap_i\future{x_i}=\bigcap_i\future{u_i}$.
  \end{lemma}
  \begin{proof}
    For all $i\in\N$, $x\in\future{u_i}$. Since $\future{u_i}$ is
    open there exists $j\in\N$ so that $x_j\in\future{u_i}$.
    This implies that $\bigcap_i\future{x_i}\subset\bigcap_i\future{u_i}$.
    As the argument above is symmetric in the two sequences we have that
    $\bigcap_i\future{x_i}=\bigcap_i\future{u_i}$.
  \end{proof}

  \begin{definition}\label{def.definitionofshadows}
    Let $(M,g)$ be a Lorentzian manifold.
    For each $x\in M$, let $(x_i)_{i\in\N}$ be a future directed
    sequence converging to $x$ and define
    \[
      \missingf{x}:={\rm Int}\left(\left(\bigcap_{i\in\N}\future{x_i}\right)\setminus\future{x}\right).
    \]
    Likewise $\missingp{x}$ is the interior of 
    $\left(\bigcap_{i\in\N}\past{y_i}\right)\setminus\past{x}$
    for any past directed sequence $(y_i)_{i\in\N}$ converging to $x$.
  \end{definition}

  Lemma \ref{lem.independenceofshadows} shows that Definition
  \ref{def.definitionofshadows} is independent of the choice of future
  (or past) directed
  sequence. 

  We think of the elements of $\missingf{x}$ as points which
  ``should'' be in $\future{x}$ but due to some global feature of causality
  are {\em missing}.
  The sets $Miss^\pm(x)$ are related to 
  Sorkin and Woolgar's $K^+$ relation
  \cite{Sorkin1996}.
%
  \begin{proposition}\label{prop_missing_conformal_invariant}
    If $(M,g)$ is a Lorentzian manifold then for all 
    $x\in M$ and all conformal factors $\Omega$,
    $\missingf{x}=\missingf{x;\Omega^2g}$.
  \end{proposition}
  \begin{proof}
    This follows directly from the identity:
    for all $y\in M$, $\future{y}=\future{y;\Omega^2g}$.
  \end{proof}
  
  \begin{example}
    The set $V$ of Example \ref{example_missing} is $\missingf{(0,0)}$.
  \end{example}

  \begin{proposition}\label{prop.missingeverydiscontin}
    Let $(M,g)$ be a Lorentzian manifold. 
    If there exists $x\in M$ so that $\missingf{x}\neq\EmptySet$
    then every conformally related Lorentzian distance, including the
    trivially conformally related distance, is
    discontinuous.
  \end{proposition}
  \begin{proof}
    Let $y\in \missingf{x}$ and
    let $(y_i=y)_{i\in\N}$ be the trivial past directed sequence converging to
    $y$ and $(x_i)$ any future directed sequence converging to $x$. 
    Since $\missingf{x}$ is open there exists $z\in \missingf{x}\cap\past{y}$.
    Hence there exists a timelike curve $\gamma\in\paths{z}{y}\cap \missingf{x}$.
    Since $z\in \missingf{x}$, for each $i\in\N$ there exists $\lambda_i\in\paths{x_i}{z}$.
    For each $i\in\N$, let $\gamma_i$ be the concatenation of
    $\lambda_i$ and $\gamma$. 
    By construction $y\not\in\future{x}$ hence $d(x,y)=0$.
    But $d(x_i,y)\geq L(\gamma_i)\geq L(\gamma)>0$. Thus the Lorentzian
    distance is discontinuous.
    The result now follows from Proposition \ref{prop_missing_conformal_invariant}.
  \end{proof}

  Proposition \ref{prop.missingeverydiscontin} has a conformal
  inverse. The conformal factor must satisfy a certain property that places
  restrictions on the global structure of the manifold. The following
  definition describes that property. The subsequent lemma describes the
  implied global structure. Proposition \ref{prop.lenghtsupnomisscontinuous}
  gives the ``conformal'' converse.

  \begin{definition}
  \label{def.lengthsuppressing}
    Let $(M,g)$ be a Lorentzian manifold.
    A function $\Omega:M\to\R$ is called a length suppressing
    conformal 
    factor if for all $\epsilon>0$ there exists $K\subset M$ a compact set
    so that if $\gamma$ is a causal curve in $M$ then
    $L(\gamma\setminus K; \Omega^2g)<\epsilon$.
  \end{definition}

  The following Lemma describes when such a conformal factor exists, 
  subtly generalising  a well-known result in the literature.
  If the manifold is strongly causal then there exists a length
  suppressing conformal factor so that the conformally related metric
  has finite diameter, \cite[Lemma 2.3]{Minguzzi2009}.
  The earliest proof of this result, known to the authors, is
  \cite[Theorem 1]{clarke1971geodesic}.
  Our result applies to non-strongly causal manifolds, e.g.\ 
  \cite[Figure 23]{Penrose1972} or \cite[Figure 7]{Minguzzi2009}, and in cases
  where acausal behaviour is restricted to a compact set, e.g.
  Example \ref{ex_null_but_constant}.

  \begin{lemma}
  \label{lem.compactexhaustionimpleslengthsuppressing}
    Let $(M,g)$ be a Lorentzian manifold.
    There exists a length suppressing conformal transformation
    if and only if for all compact exhaustions, $(K_i)_{i\in\N}$,
    of $M$ there exists $N\in\N$ so that $i\geq N$ implies that 
    there exists $k_i\in\R^+$ so that for all causal curves
    $\gamma$, $L\left((\gamma\cap K_i)\setminus K_N\right)\leq k_i$.
  \end{lemma}
  \begin{proof}
    Suppose that there exists a
    compact exhaustion, $(K_i)_{i\in\N}$,
    and $N\in\N$ so that $i\geq N$ implies that 
    there exists $k_i\in\R^+$ so that for all causal curves
    $\gamma$, $L\left((\gamma\cap K_i)\setminus K_N\right)\leq k_i$.
    Let $\Omega_i:M\to\R$ be such that 
    $\Omega_i|_{K_i\setminus\interior{K_{i-1}}}=\frac{1}{2^{i}k_i}$,
    $\Omega_i|_{K_{i-2}\cup\left(M\setminus\interior{K_{i+1}}\right)}=1$ and
    $\Omega_i(K_{i+1})\subset [\frac{1}{2^{i}k_i},1]$.
    Let $\Omega=\prod_{i=1}^\infty\Omega_i$. This product is well-defined as
    for any $x\in M$, there exists $i\in\N$ so that
    $x\in K_{i}\setminus\interior{K_{i-1}}$. In this case 
    $\Omega(x)=\Omega_{i+1}(x)\Omega_{i}(x)\Omega_{i-1}(x)\leq \Omega_i(x)$.
    With this definition $i\geq N$ implies that
    $L\left((\gamma\cap K_i)\setminus K_N;\Omega^2g\right)
      \leq L\left((\gamma\cap K_i)\setminus K_N;\Omega_i^2g\right)<\frac{1}{2^i}$.
    Choose $\epsilon>0$ and $n>N$ so that 
    $\sum_{i=n}^\infty\frac{1}{2^i}\leq \epsilon$.
    Let $\gamma$ be a causal curve then
    \begin{align*}
      L(\gamma\setminus K_n;\Omega^2g)&\leq\sum_{i=n}^\infty
        L\left(\left(\gamma\cap K_i\right)\setminus K_n;\Omega^2g\right)
        <\sum_{i=n}^\infty\frac{1}{2^i}\leq\epsilon,
    \end{align*}
    as required.

    Suppose that exists a length suppressing conformal factor, $\Omega$.
    Let $(K_i)_{i\in\N}$ be a compact
    exhaustion of $M$. Choose $\epsilon>0$. Then there
    exists $K\subset M$ compact so that
    $L(\gamma\setminus K;\Omega^2g)<\epsilon$. Choose $N\in\N$
    so that $K\subset K_N$ and let $i\geq N$.
    Since $K_i$ is compact there exists $h_i\in\R^+$
    so that $0<h_i=\min\{\Omega(x):x\in K_i\}$.
    Let $\gamma$ be a causal curve then
    \begin{align*}
      \epsilon&>L\left(\gamma\setminus K_N;\Omega^2g\right)
          \geq L\left((\gamma\cap K_i)\setminus K_N;\Omega^2g\right)
        =h_i L\left((\gamma\cap K_i)\setminus K_N;g\right).
    \end{align*}
    Hence 
    $L\left((\gamma\cap K_i)\setminus K_N;g\right)<\frac{\epsilon}{h_i}$,
    and we may take $k_i=\frac{\epsilon}{h_i}$ to get the result.
  \end{proof}

  \begin{corollary}
    Let $(M,g)$ be a Lorentzian manifold,
    let $(K_i)_{i\in\N}$ be a compact exhaustion
    of $M$ so that 
    there exists $N\in\N$ so that $i\geq N$ implies that 
    there exists $k_i\in\R^+$ so that for all causal curves
    $\gamma$, $L\left((\gamma\cap K_i)\setminus K_N\right)\leq k_i$
    and
    let $\Omega$ be a length suppressing conformal factor.
    If
    $\max\{d(x,y):x,y\in K_N\}<\infty$ then
    $(M, \Omega^2g)$ has finite diameter.
  \end{corollary}
  \begin{proof}
    Let $k=\max\{d(x,y):x,y\in K_N\}$.
    Choose $\epsilon>0$. By construction there exists a compact set
    $K$ so that for all timelike curves $\gamma$, 
    $L(\gamma\setminus K;\Omega^2g)<\epsilon$.
    There exists $i\geq N$ so that $K\subset K_i$.
    In this case 
    $L\left((\gamma\cap K_i)\setminus K_N\right)\leq k_i$.
    Since $\Omega$ is continuous there exists $w>0$ so that
    for all $x\in K_i$ $\Omega(x)<w$.
    Thus
    \[
      L(\gamma;\Omega^2g)\leq
        L(\gamma\setminus K;\Omega^2g)
        + L\left((\gamma\cap K_i)\setminus K_N;\Omega^2g\right)
        + L(\gamma\cap K_N;\Omega^2g)
        \leq \epsilon + wk_i + wk.
    \]
    As this holds for any timelike curve the diameter of
    $(M,\Omega^2g)$ is less than or equal to 
    $\epsilon + wk_i + wk$.
  \end{proof}

  We can now present the ``conformal'' converse of Proposition
  \ref{prop.missingeverydiscontin}. This is a generalisation
  of \cite[Lemma 2.3]{Minguzzi2009} and 
  \cite[Theorem 1]{clarke1971geodesic}.

  \begin{proposition}
  \label{prop.lenghtsupnomisscontinuous}
    Let $(M,g)$ be a Lorentzian metric.
    If $\Omega:M\to\R$ is a length suppressing conformal factor
    and for all $x\in M$, $\missingf{x}=\EmptySet=\missingp{x}$
    then $(M,\Omega^2g)$ has continuous Lorentzian distance.
  \end{proposition}
  \begin{proof}
    Let $x,y\in M$ and suppose that there exists
    $(x_i)_{i\in\N}\subset M$ a future directed sequence
    converging to $x$, $(y_i)_{i\in\N}\subset M$ a past
    directed sequence converging to $y$, and
    suppose that there exists a sequence of curves
    $(\gamma_i)_{i\in\N}$ so that for all $i\in\N$
    $\gamma_i\in\Omega_{x_i,y_i}$.
    If it were the case that
    $\lim_{i\to\infty}\len{\gamma_i\setminus\future{x}}=0$ and
    $\lim_{i\to\infty}\len{\gamma_i\setminus\past{x}}=0$ then,
    as $x,y\in M$ are arbitrary,
    Theorem \ref{thm.lddis} would give the result. We now
    show that the limits of the particular
    curve lengths are in fact equal to $0$.

    Let $h=\Omega^2g$.
    Choose $\epsilon>0$ and $K\subset M$ a compact set so that
    if $\lambda$ is a timelike curve in $M$ then
    $L(\lambda\setminus K; h)<\epsilon$.
    Such a compact set exists since we assume that a 
    length suppressing conformal factor exists.
    We know that for each $i\in\N$, 
    \begin{align*}
      L(\gamma_i\setminus\future{x};h) &= 
      L\left(\left(\gamma_i\setminus\future{x}\right)\cap K; h\right) + 
      L\left(\left(\gamma_i\setminus\future{x}\right)\setminus K; h\right) \\
      &< L\left(\left(\gamma_i\setminus\future{x}\right)\cap K; h\right) + \epsilon.
    \end{align*}
    Suppose, for a contradiction, 
    that there exists a subsequence, $(\gamma_{k_i})$, 
    of $(\gamma_i)$ such that 
    \[
      \lim_{i\to\infty}
        L\left(\left(\gamma_{k_i}\setminus\future{x}\right)\cap K; h\right)\neq 0.
    \]
    Lemma \ref{CurLim:Lem.CurveUniformConvergenceInBoundedRegion}
    implies the existence of a limit curve, $\gamma$, in $K$
    and Lemma \ref{lem.uniformconvergenceimplieslengthconvergence}
    implies that
    $0\neq\lim_{i\to\infty}
      L\left(\left(\gamma_{k_i}\setminus\future{x}\right)\cap K; h\right)
      \leq L(\gamma)$.
    This implies that $\gamma$ has some timelike subcurve $\mu$. 
    By definition
    for all $i\in\N$,
    $\gamma_i\subset\future{x_i}$. Thus $\gamma\subset\bigcap_{i}\future{x_i}$.
    By construction $\gamma\not\subset\future{x}$, hence
    $\gamma\subset\left(\bigcap_i\future{x_i}\right)\setminus\future{x}$.
    Let $p,q\in\mu\subset\gamma$ so that $q\in\future{p}$. Then as 
    $p,q\in\left(\bigcap_i\future{x_i}\right)\setminus\future{x}$
    We know that $\EmptySet\neq\future{p}\cap\past{q}\subset
      \left(\bigcap_i\future{x_i}\right)\setminus\future{x}$.
    Therefore $\missingf{x}\neq\EmptySet$.
    This is a contradiction and hence
    $\lim_{i\to\infty}
      L\left(\left(\gamma_{k_i}\setminus\future{x}\right)\cap K; h\right)=0$.

    We now know that 
    $\lim_{i\to\infty}L(\gamma_i\setminus\future{x};h) < \epsilon$.
    Since $\epsilon$ was arbitrary, we see that in fact
    $\lim_{i\to\infty}L(\gamma_i\setminus\future{x};h) = 0$.
    The time reverse of the above arguments shows that
    $\lim_{i\to\infty}L(\gamma_i\setminus\past{y};h) = 0$ also, as required.
  \end{proof}

  We summarise the results of this section.
  \begin{theorem}\label{theorem_not_fintie_cts_imply_causal_hole}
    Let $(M,g)$ be a Lorentzian manifold.
    If the Lorentzian distance is either
    \begin{enumerate}
      \item infinite and $M$ is chronological, or
      \item discontinuous, 
    \end{enumerate}
    then there exists $x\in M$, $y\in\future{x}$ and an
    inextendible incomplete past directed
    timelike curve so that
    for all $u\in\past{x}$, $\lambda\subset\future{u}$
    and $\future{y}\subset\future{\lambda}$.
  \end{theorem}
  \begin{proof}
    This follows from Propositions
    \ref{prop.discontinuousimpliesnoncompact}
    and \ref{finite.infintiesimpliesnoncompact}
  \end{proof}

  \begin{theorem}\label{thm_causal_hole_implies_conformally}
    If there exists $x\in M$, $y\in\future{x}$ and an
    inextendible incomplete past directed
    timelike curve so that
    for all $u\in\past{x}$, $\lambda\subset\future{u}$
    and $\future{y}\subset\future{\lambda}$,
    then 
    \begin{enumerate}
      \item the Lorentzian distance is conformally infinite, and  
      \item if either there exists $u\in\past{x}$ and
        $v\in\future{y}$ so that $d(u,v)<\infty$ or
        if $d(x,y)=0$ then the Lorentzian distance is
        conformally discontinuous.
    \end{enumerate}
  \end{theorem}
  \begin{proof}
    Lemma \ref{lem.inftydivergence} proves the ``not finite'' case.
    Proposition \ref{prop.noncompactimpliesdiscontinuous}
    and
    Lemma \ref{lem.inftydiscontinuousonachronal} prove the
    ``not continuous'' case.
  \end{proof}

  The need for $M$ to be chronological in
  Theorem \ref{thm_causal_hole_implies_conformally}
  is necessary. The needed
  example is given by the Misner spacetime, \cite[Page 171]{HawkingEllis1975}.
  The additional conditions needed in the ``discontinuous''
  case of Theorem \ref{thm_causal_hole_implies_conformally}
  are also necessary.
  The counter example is provided by a totally vicious manifold
  with a point removed.
  
  \subsection{Immediate applications}
  \label{section_conformal_applications}
  
    The results of Section \ref{sec.conformal} provide a collection
    of tools that can be put to use to prove interesting results.
    For example,

    \begin{theorem}
      \label{theorem_strong_causal_conformal}
      If $(M,g)$ is a strongly causal manifold
      then there exists a conformal transformation $\Omega$
      so that $(M,\Omega^2g)$ has finite Lorentzian distance.
      Either
      \begin{enumerate}
        \item there exists $x$ so that $\missingf{x}\neq\EmptySet$ or
          $\missingp{x}\neq\EmptySet$, or
        \item $\Omega$ can be chosen such that 
        $(M,\Omega^2g)$ also has continuous Lorentzian distance.
      \end{enumerate}
    \end{theorem}
    \begin{proof}
      By definition every strongly causal manifold has a compact exhaustion
      that satisfies the conditions of 
      Lemma \ref{lem.compactexhaustionimpleslengthsuppressing},
      see the proof of \cite[Lemma 2.3]{Minguzzi2009}.
      Hence
      the manifold carries a length suppressing conformal factor.
      Corollary \ref{cor.buildaneikonalfunction}
      and
      Proposition \ref{prop.lenghtsupnomisscontinuous} 
      prove the result.
    \end{proof}

    Theorem \ref{theorem_strong_causal_conformal} is a generalisation
    of \cite[Theorem 2.4]{Minguzzi2009}.
    Next we relate the sets $\missingfp{x}$
    outer continuity of $I^\pm$.

    \begin{proposition}\label{prop.causallycontinuouspartialconverse}
      Let $M$ be a Lorentzian manifold.
      The set-valued function $I^+$ is outer continuous if and only if
      for all $x\in M$, $\missingf{x}=\EmptySet$.
    \end{proposition}
    \begin{proof}
      Suppose that there exists $x\in M$ so that $\missingf{x}\neq\EmptySet$.
      Since $\missingf{x}$ is open it contains a compact set. The definition of 
      $\missingf{x}$ now implies that the manifold is not causally continuous.

      Suppose that $M$ is not causally continuous. Then there exists $x\in M$
      and a compact set $K\subset M\setminus\closure{\future{x}}$ so that
      for all neighbourhoods $U$ of $x$ there exists $y\in U$
      so that $K\cap \closure{\future{y}}\neq\EmptySet$.
      Let $(x_i)_{i\in\N}$ be a future directed sequence converging to $x$.
      By assumption for each $i$ there exists $k_i\in 
        K\cap\closure{\future{x_i}}$. As the sequence $(k_i)_{i\in\N}$
      lies in the compact set $K$ some subsequence has a limit point
      $k\in K$. Since $K\subset M\setminus\closure{\future{x}}$
      there exists an open neighbourhood $U$ of $k$ so that
      $U\subset M\setminus\closure{\future{x}}$. Let 
      $V\subset U\cap\future{k}$.
      Since $(x_i)$ is future directed 
      for all $j\geq i$, $k_j\in\future{x_i}$.
      Hence for all $i\in\N$, $k\in\closure{\future{x_i}}$,
      so for all $i\in\N$,
      $V\subset\future{x_i}$. 
      By construction $\EmptySet\neq V\subset\missingf{x}$, as required.
    \end{proof}

    Proposition \ref{prop.causallycontinuouspartialconverse} 
    has the implication that if the manifold is mildly well-behaved
    causally then the sets, $\missingfp{x}$,
    are related to the very strong causality condition, causal continuity.

    \begin{theorem}\label{lem.missnotcc}
      Let $(M,g)$ be a distinguishing Lorentzian manifold.
      The manifold $M$ is causally continuous
      if and only if for all $x\in M$, $\missingf{x}=\EmptySet=\missingp{x}$.
    \end{theorem}
    \begin{proof}
      This follows from Proposition \ref{prop.causallycontinuouspartialconverse}
      and the definition of causal continuity.
    \end{proof}
  
    We also obtain a new proof of a known relation between continuity of the
    Lorentzian distance and causality in the distinguishing case.
  
    \begin{theorem}[Theorem 4.24 of \cite{BeemEhrlichEasley1996}]
      \label{thm.continuousimpliescausallycontinuous}
      Let $(M,g)$ be a distinguishing Lorentzian manifold.
      If the Lorentzian distance is continuous then the manifold is
      causally continuous.
    \end{theorem}
    \begin{proof}
      This follows from
      Proposition \ref{prop.missingeverydiscontin} 
      and Theorem \ref{lem.missnotcc}.
    \end{proof}

    It is possible to generalise
    \cite[Theorem 2.4]{Minguzzi2009} further.
  
    \begin{theorem}\label{thm.generalisedMin}
      Let $(M,g)$ be a distinguishing 
      Lorentzian manifold and let $\Omega$ be
      a length suppressing conformal factor.
      The manifold is causally continuous if and only if
      $(M,\Omega^2g)$ has continuous Lorentzian distance.
    \end{theorem}
    \begin{proof}
      Theorem
      \ref{lem.missnotcc}
      and Proposition 
      \ref{prop.lenghtsupnomisscontinuous} show that
      causal continuity implies that $(M,\Omega^2g)$
      has continuous Lorentzian distance.
      Proposition \ref{prop.missingeverydiscontin} and
      Theorem \ref{lem.missnotcc}  show that
      if $M$ is distinguishing and $(M,\Omega^2g)$ has
      a continuous Lorentzian distance
      then $M$ is causally continuous.
    \end{proof}

    Our techniques also give a short proof of the 
    following well-known implication.
    \begin{theorem}[Lemma 4.5 of \cite{BeemEhrlichEasley1996}]
    \label{thm.gh}
      If $(M,g)$ is globally hyperbolic then the Lorentzian distance is finite
      and continuous.
    \end{theorem}
    \begin{proof}
      By definition, for all $u,v\in M$, $\causalfuture{u}\cap\causalpast{v}$
      is compact. 
      Proposition \ref{prop.discontinuousimpliesnoncompact} implies that
      the Lorentzian distance is continuous.
      Lemmas \ref{lem.uniformconvergenceimplieslengthconvergence}
      and \ref{CurLim:Lem.CurveUniformConvergenceInBoundedRegion}
      imply that the Lorentzian distance is finite.
    \end{proof}

\section{Finiteness, continuity and the causal hierarchy}
\label{sec:fin-cts-caus}

  The previous sections of this paper should give the reader the impression
  that the finiteness and continuity of the Lorentzian distance do not connect
  well with rungs in causal hierarchy. 
  In this section we exactly describe the relationship between finiteness,
  continuity and the causal hierarchy.
  We consider here a collection of standard
  causality conditions,
  \cite{Minguzzi2009}, whose relations are described in Figure
  \ref{fig_causality}.

  \begin{figure}[htb]
    \begin{center}
      \begin{tikzpicture}
        \node (gh) {\footnotesize Globally hyperbolic};
        \node (cs) [right=of gh] {\footnotesize Causally simple};
        \node (cc) [right=of cs] {\footnotesize  Causally continuous};
        \node (stab) [below=of gh] {\footnotesize Stably causal};
        \node (stro) [below=of cs] {\footnotesize Strongly causal};
        \node (d) [below=of cc] {\footnotesize Distinguishing};
        \node (cau) [below=of stab] {\footnotesize Causal};
        \node (chro) [below=of stro] {\footnotesize Chronological};
        \node (tv) [below=of d] {\footnotesize Totally viscious};
        \draw[double, ->] (gh) -- (cs);
        \draw[double, ->] (cs) -- (cc);
        \draw[double, ->] (cc.east) .. controls +(3,-1)
          and +(-3,1) .. (stab.west);
        \draw[double, ->] (stab) -- (stro);
        \draw[double, ->] (stro) -- (d);
        \draw[double, ->] (d.east) .. controls +(3.25,-1.25)
          and +(-3.25,1.25) .. (cau);
        \draw[double, ->] (cau) -- (chro);
        \draw[double, ->] (chro) -- (tv);
      \end{tikzpicture}
    \end{center}
    \caption{The causality conditions appearing in Section
      \ref{sec:fin-cts-caus}. An arrow, {\protect\tikzRightarrow},
      from one condition to another indicates that the first
      condition implies the second. For example, global hyperbolicity
      implies causally simple.}
    \label{fig_causality}
  \end{figure}
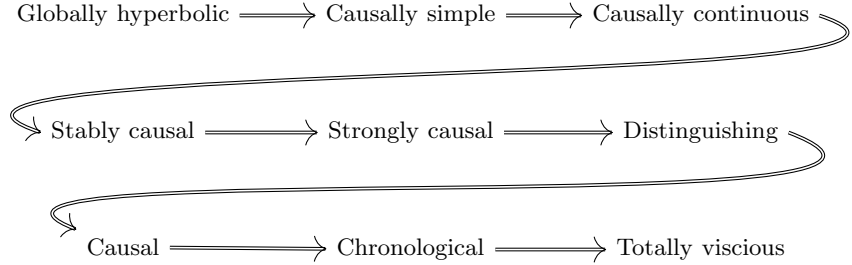

  \begin{theorem}
    \label{thm:independence}
    Let $(M,g)$ be a Lorentzian manifold, and $d$ the Lorentzian 
    distance function. 
    \begin{enumerate}
      \item The condition ``$d$ is finite and continuous'' is independent of each
        of the following
        causality conditions on $(M,g)$: causally simple, causally continuous,
        stably causal, strongly causal, distinguishing, causal and chronological.

      \item The condition ``$d$ is finite'' is independent of each of the 
        following causality
        conditions on $(M,g)$: causally simple, causally continuous, stably causal,
        strongly causal, distinguishing, causal. The condition ``$d$ is finite''
        implies the condition ``chronological'', while totally vicious implies 
        ``$d$ is not finite''.

      \item The condition ``$d$ is continuous'' is independent of each
        of the following causality
        conditions on $(M,g)$: 
        causally simple, causally continuous, stably causal, strongly
        causal, distinguishing, causal and chronological. Totally
        vicious implies ``$d$ is continuous''.

      \item More precisely, for each of the pairs of causal conditions (not A, B),
        \label{detailed_indepedence}
        \begin{itemize}
          \item not causally simple, causally continuous
          \item not stably causal, strongly causal
          \item not strongly causal, distinguishing
          \item not distinguishing, causal
          \item not causal, chronological
        \end{itemize}
        there exists four Lorentzian manifolds satisfying 
        $B$ but not $A$ and such that
        the Lorentzian distance is respectively
        \begin{itemize}
          \item finite and continuous,
          \item finite and discontinuous,
          \item infinite and continuous,
          \item infinite and discontinuous.
        \end{itemize}

        The condition (not causally continuous, stably causal) 
        implies that 
        every conformally related manifold is discontinuous
        and there
        exists two Lorentzian manifolds,
        satisfying the condition, so that the
        Lorentzian distance is finite and
        infinite respectively.
    \end{enumerate}
  \end{theorem}
  The proof of this theorem is a collection of examples and counterexamples
  divided into nine cases
  based on the  
  causality conditions:
  causally simple,
  causally continuous,
  stably causal,
  strongly causal,
  distinguishing,
  causal,
  chronological, and
  totally vicious.
  We include 
  globally hyperbolic in the discussion below for completeness.

  \subsection{Totally vicious}

    In a totally vicious manifold the Lorentzian distance,
    considered as an extended function $d:M\times M\to[0,\infty]$, is
    constant (infinite) and therefore continuous. Finiteness of the Lorentzian distance
    clearly does not hold.
    Thus any totally vicious manifold is an example of a continuous but
    not finite manifold. This proves the last statements of 
    items $2$ and $3$ in Theorem \ref{thm:independence}.

  \subsection{Chronological}\label{sec_chronological}
    
    Example \ref{ex_null_but_constant}, below,
    presents a chronological non-causal manifold
    with finite and continuous Lorentzian distance.
    
    By removing a vertical (parallel to $\R$)  
    line segment 
    from the manifold in 
    Example \ref{ex_null_but_constant}, points $x,y$,
    so that $\missingp{x}\neq\EmptySet$
    and $\missingp{y}\neq\EmptySet$ are introduced.
    Thus the resulting manifold will be
    chronological non-causal with finite but discontinuous
    Lorentzian distance, by Proposition \ref{prop.missingeverydiscontin}.

    By removing any point not on the closed null 
    curve (in order to preserve non-causality)
    from the manifold given in Example
    \ref{ex_null_but_constant} and applying
    Lemma \ref{lem.inftydivergence}, we obtain a conformal transformation
    which produces a 
    chronological non-causal manifold with infinite and discontinuous
    Lorentzian distance. Proposition \ref{prop.noncompactimpliesdiscontinuous}
    gives the proof of discontinuity. 
    
    Producing 
    a chronological non-causal manifold with infinite and continuous
    distance is slightly more awkward since we need to avoid, for example,
    the situations described by Proposition 
    \ref{prop.noncompactimpliesdiscontinuous} and Lemma
    \ref{lem.inftydiscontinuousonachronal}.
    Example \ref{example_fin_dis_causal} provides the details.
    
    This proves all statements 
    about chronology in Theorem \ref{thm:independence}.

    Note also that
    Examples \ref{example_missing} 
    and \ref{example.weaksingularity} present 
    chronological manifolds
    with discontinuous but finite distance, where the discontinuities
    arise via different mechanisms.
    Example \ref{example.infinite}
    gives a chronological manifold with continuous but
    not finite distance.

    \begin{example}\label{ex_null_but_constant}
      Let $M=S^1\times\R$ equipped with the metric
      \[
        g = -s(t)dt^2 + 2\sqrt{1-s(y)^2}dtd\theta + s(t)d\theta^2.
      \]
      Let $s(t)=\arctan^2(t)$ so that for large positive and
      negative $t$ the metric is approximately $-dt^2 + d\theta^2$,
      whereas for $t$ close to $0$ the metric is approximately
      $2dtd\theta$. Thus light cones ``tip over'' close to the  ``waist''
      $\{(x,y):y=0\}$, which is a closed null curve.

      Let $\gamma(\tau)=(t(\tau),\theta(\tau))\in M$ 
      be a geodesic with affine parameter $\tau$ so that
      if $\epsilon\in\{-1,0,1\}$ then
      \[
        g(\gamma',\gamma')=-s(t(\tau))(t'(\tau))^2 
          - 2\sqrt{1-s(t(\tau))^2}t'(\tau)\theta'(\tau)
          + s(t(\tau))(\theta'(\tau))^2=\epsilon.
      \]
      Since the coefficients of the metric do not depend on $\theta$ 
      there is a second constant of integration, $q\in\R$, given by
      \[
        q = -\sqrt{1-s(t(\tau))}t'(\tau) + s(t(\tau))\theta'(\tau).
      \]
      By substituting the equation involving $\epsilon$ into the square
      of the equation involving $q$, the following equation for $t'(\tau)$
      can be derived
      \[
        \left(t'(\tau)\right)^2 + \epsilon s(t(\tau))= q^2.
      \]
      This equation has the implicit solution,
      when at least one of $q$ and $\epsilon$ is non-zero,
      \begin{equation}\label{ex.elegantmisnerydahs}
        \int_{t(K)}^{t(\tau)}\frac{1}{\sqrt{q^2-\epsilon s(\xi)}}\dd\xi=K\pm \tau,
      \end{equation}
      where $K$ is a constant of integration that can be taken to be
      $0$ since $\tau$ is affine. 

      By a standard cutting the corner argument,
      \cite[Page 7.6 and Definition 2.13]{Penrose1972},
      for any $(t,\theta)\in M$, $t<0$, the longest timelike curve from
      $(t,\theta)$ to any point $(0, \phi)$ will be a curve which approaches
      but never reaches the $t=0$ surface.
      Since we can arrange for a sequence of length maximising curves that
      approach but never reach the waist to be contained in a compact subset
      of $M$ the limit geodesic will exist, 
      Lemma \ref{CurLim:Lem.CurveUniformConvergenceInBoundedRegion}.
      This limit curve will be a timelike geodesic, Lemma 
      \ref{lem.uniformconvergenceimplieslengthconvergence}.
      Hence there exists a timelike geodesic from
      $(t,\theta)$ that approaches but does not reach the waist.

      Let $\gamma(\tau)=(t(\tau),\theta(\tau))$ 
      be a
      unit length,
      future directed, timelike geodesic from $(a,b)$, $a<0$, so that
      for all $\tau\in\domain{\gamma}$, $t(\tau)<0$ and $\gamma$ winds around
      the waist at $t=0$. 
      This implies that $\gamma$ is inextendible,
      $t'(\tau)>0$ and that
      ${t'(\tau)}\to 0$ as $t(\tau)$ approaches $0$.
      Equation \eqref{ex.elegantmisnerydahs} implies that
      $\tau$ has a maximum, which
      we denote by $\tau_\gamma$, 
      and so $\gamma$ is incomplete and inextendible.
      For $\epsilon=-1$,
      Equation \eqref{ex.elegantmisnerydahs} implies that the maximum
      of $\tau$ will occur when $q=0$. This implies that
      $\theta'(K)$ is given by solving
      $\sqrt{1-s(t(K)^2)}t'(K) + s(t(K))\theta'(K)=0$.
      Hence the Lorentzian distance from a level $t$ surface
      to the waist is independent of $\theta$.
      By definition of an affine parameter
      $L(\gamma)=\tau_\gamma - K$.
      Since equation \eqref{ex.elegantmisnerydahs} has smooth dependence
      on the domain of integration and as the metric is symmetric about 
      the $t=0$ surface we know that the Lorentzian distance
      is continuous in a neighbourhood of the $t=0$ surface.
      It is clear that the Lorentzian distance is continuous everywhere
      else in $M$ as both the future and past of the $t=0$ surface
      are globally hyperbolic submanifolds.
      Hence the Lorentzian distance is continuous on all of $M$.
    \end{example}

    \begin{example}\label{example_fin_dis_causal}
      Let $(M,g)$ be the manifold of Example \ref{ex_null_but_constant}.
      Choose $(\tau,s)\in M$, $\tau>0$, and let $U, V\subset M$ 
      be open neighbourhoods,
      with compact closure,
      of $(\tau,s)$ so that 
      $U\subset\future{\left\{(t,\theta)\in M:\,t=0\right\}}$,
      $\overline{V}\subset U$ and $U$ is 
      homeomorphic to the $2$-dimensional
      ball.
      Choose $\rho:M\to[0,1]$ with support in $U$ so that $\rho|_V=1$.
      Define $\Omega:M\setminus\{(\tau,s)\}\to\R$ by
      $\Omega(t,\theta)=\frac{c}{(t-\tau)^2 + (\theta-s)^2}$,
      where $c\in\R$ is chosen so that 
      $\min\{\Omega(p):p\in U\} > 2$
      Let $O=M\setminus\{(\tau,s)\}$ and equip
      $O$ with the metric $h=(1 + \rho(\Omega - 1))^2g$.

      The manifold $(N,h)$ 
      is chronological and not causal since there are no closed
      timelike curves but there is a closed null curve, i.e.\ the
      surface $\{(t,\theta)\in N:t=0\}$.
      The Lorentzian distance induced by $h$
      is
      infinite and continuous for the same reasons
      that the Lorentzian distance of Example \ref{example.infinite}
      is infinite and continuous.
    \end{example}

  \subsection{Causal}

    By construction there is only one closed null curve in
    the manifold of Example \ref{ex_null_but_constant}. In addition
    there are no closed timelike curves.
    Hence if a point is removed from the closed null curve
    the manifold will be causal.
    Let $(0,\theta_1)$ and $(0,\theta_2)$, $\theta_1\neq\theta_2$,
    be in the manifold. Then $\future{(0,\theta_1)}=\future{(0,\theta_2)}$.
    Removing a point from the closed null curve does not effect this
    set equivalence.
    Thus the manifold of Example \ref{ex_null_but_constant} with
    a point removed from the closed null curve
    is causal but not distinguishing. In particular this
    new manifold has finite and continuous Lorentzian distance.

    To build a causal, not distinguishing, Lorentzian manifold
    with infinite and continuous Lorentzian distance we can
    use the technique given in Example \ref{example_fin_dis_causal}.
    Start with the manifold of Example \ref{ex_null_but_constant}
    remove a point on the closed null curve and
    apply a conformal transformation built as in
    Example \ref{example_fin_dis_causal}.
    
    Removing a vertical line from Example
    \ref{ex_null_but_constant}, as in Subsection
    \ref{sec_chronological} and removing a point from the closed null
    curve, produces a causal but not distinguishing
    manifold with finite and discontinuous.

    Removing a vertical line from Example
    \ref{example_fin_dis_causal}, as in Subsection
    \ref{sec_chronological} and removing
    a point from the closed null curve, produces a causal but not distinguishing
    manifold with infinite and discontinuous.

    This proves all statements 
    about causality in Theorem \ref{thm:independence}.
    
    Examples \ref{example_missing},
    \ref{example.weaksingularity},
    and \ref{example.infinite} apply in this case too.

  \subsection{Distinguishing}\label{sec_distinguishgin}

    Example \ref{ex.nonsc_level_surfaces} below
    presents a distinguishing non-strongly
    causal manifold with finite but discontinuous distance.

    Because only horizontal ``half''-infinite lines have been removed from
    $[-4,4]\times\R$ to obtain the manifold $N$ of
    Example \ref{ex.nonsc_level_surfaces}, for
    all $x\in N$ we have 
    $\missingp{x}=\missingf{x}=\EmptySet$.
    Hence we
    can apply Proposition \ref{prop.lenghtsupnomisscontinuous}
    to Example \ref{ex.nonsc_level_surfaces} 
    to get a 
    distinguishing non-strongly causal
    manifold with finite and continuous distance

    Applying Lemma \ref{lem.inftydivergence} produces
    a distinguishing non-strongly
    causal manifold with infinite but discontinuous distance.

    Production of a
    distinguishing non-strongly
    causal manifold with infinite and continuous distance
    is more complicated, but can be
    achieved by following the method of
    Example \ref{example_fin_dis_causal}, see
    Example \ref{example_nonsc_infinte_cts}

    This proves all of the statements about the 
    distinguishing case in Theorem \ref{thm:independence}.

    Examples \ref{example_missing},
    \ref{example.weaksingularity},
    and \ref{example.infinite} apply in this case too.

    \begin{example}\label{ex.nonsc_level_surfaces}
      Let $N=\left([-4,4]\times\R\right)
        \setminus
        \left(\{(1,x): x\leq 1\}\cup\{(-1,x):x\geq -1\}\right)$.
      Define two points $(t,x),(s,y)\in N$ to be equivalent, 
      $(t,x)\sim(s,y)$, if and only if
      $x=y$ and
      $t=\pm s$. Let $M=N/\sim$ be the quotient manifold: see Figure
      \ref{fig-8}. Equip $N$ with the
      metric $-dt^2+dx^2$. This induces a metric on $M$. 
      This is a standard example of a distinguishing 
      non-stably causal manifold, e.g.\
      Figure 38 of \cite[Page 193]{HawkingEllis1975}.

      Let $\epsilon_1,\epsilon_2>0$ and let 
      $\gamma:[0,3]\to M$ be 
      the curve given by 
      \[
        \gamma(t)=\left\{\begin{aligned}
          & (0, \epsilon_1) + t(1,1) && t\in[0,1)\\
          & (1, 1+\epsilon_1) + (t-1)(-2,-2-\epsilon_1-\epsilon_2) 
            && t\in[1,2)\\
          & (-1, -1 - \epsilon_2) + t(1,1).
        \end{aligned}\right.
      \]
      For $\epsilon_1$ and $\epsilon_2$ small enough, the curve 
      $\gamma$
      will be timelike. 

      Choose $U$ an open neighbourhood about $(0,0)$.
      By taking $\epsilon_1$ and $\epsilon_2$
      arbitrarily small we can see that $\gamma$ starts in $U$
      and returns to it. Thus $M$ is not strongly causal.
      The manifold is distinguishing, this can be checked directly.
    \end{example}

    \begin{figure}
      \label{figure_nonsc_cylinder}.
      \begin{center}
        \begin{tikzpicture}
          \draw[line width=1.5pt, smooth, samples=100, domain=-4:4, variable=\x] 
            plot (\x, 2);
          \draw[line width=1.5pt, smooth, samples=100, domain=-4:4, variable=\x] 
            plot (\x, -2);

          \draw[line width=1.5pt, smooth, samples=100, domain=-4:0.8, variable=\x] 
            plot (\x, 0.8);
          \draw[line width=1.5pt, smooth, samples=100, domain=-0.8:4, variable=\x] 
            plot (\x, -0.8);
          \draw[line width=1pt, dashed, smooth, samples=100, domain=-0.8:0.8, variable=\x] plot (\x,\x);

          \draw[line width=1.0pt, smooth, samples=100, domain=-0.2:0.2, variable=\x] 
            plot (\x-0.1,\x+2);
          \draw[line width=1.0pt, smooth, samples=100, domain=-0.2:0.2, variable=\x] 
            plot (\x+0.1,\x+2);
          \draw[line width=1.0pt, smooth, samples=100, domain=-0.2:0.2, variable=\x] 
            plot (\x-0.1,\x-2);
          \draw[line width=1.0pt, smooth, samples=100, domain=-0.2:0.2, variable=\x] 
            plot (\x+0.1,\x-2);

          \filldraw (-0.8,-0.8) circle (2.5pt);
          \node at (-1,-1.2) {$(-1,-1)$};
        \end{tikzpicture}
      \end{center}
      \caption{An illustration of 
        Example \ref{ex.nonsc_level_surfaces}.
        The thick lines without two cross hatches have been removed
        from the manifold. The thick lines with cross hatches have
        been identified. The dashed line is a null surface
        of geometric importance. The black dot marks the $(-1,-1)$ point.
        }
        \label{fig-8}
    \end{figure}
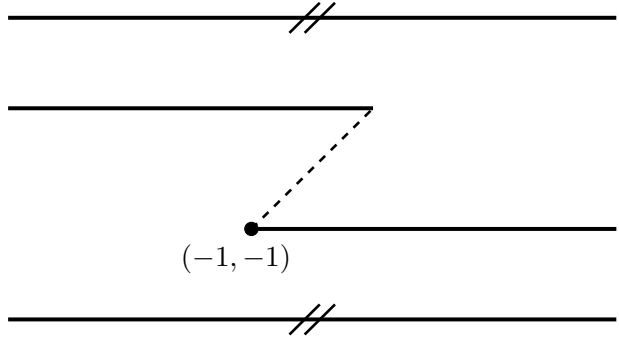

    \begin{example}\label{example_nonsc_infinte_cts}
      Let $N$ and $M$ be as in Example \ref{ex.nonsc_level_surfaces}.
      Let $U,V\subset \R^2$ be pre-compact open neighbourhoods
      of $(-1,-1)\in\R^2$ so that $\overline{V}\subset U$ and
      \[
        U\setminus
        \left(\{(1,x): x\leq 1\}\cup\{(-1,x):x\geq -1\}\right)\subset N.
      \]
      Let $\rho:\R^2\to[0,1]$ be a bump function
      with support in $U$ and such that $\rho|_V=1$.
      Let $\Omega:\R^2\to\R$ be defined by
      $\Omega(t,x) = \frac{c}{(t+1)^2 + (x+1)^2}$,
      where $c$ is chosen so that
      $\min\{\Omega(p):p\in U\}\geq 2$. Define
      the metric $h$ on $M$ by restriction of the metric
      $(1+\rho(\Omega -1))^2(-dt^2_dx^2)$ on $N$.
      The manifold $(M, h)$ is distinguishing but not
      strongly causal, as causal structure is conformally invariant.
      The induced Lorentzian distance is infinite and continuous
      by the same arguments used in Example \ref{example.infinite}.
    \end{example}

  \subsection{Strongly causal}

    Example \ref{example_strong_not_stab} gives
    a strongly causal but not stably causal manifold
    with finite and discontinuous Lorentzian distance.

    Since the lines that have been removed to produce the manifold
    $N$ in
    Example \ref{example_strong_not_stab} are spacelike and ``half''-infinite
    we know that for all $x\in M$, where $M$ is defined as in
    Example \ref{example_strong_not_stab}, $\missingp{x}=\missingf{x}=\EmptySet$.
    Hence we
    can applying Proposition \ref{prop.lenghtsupnomisscontinuous}
    to Example \ref{example_strong_not_stab}
    to get a 
    strongly causal non-stably causal
    manifold with finite and continuous distance.

    Applying Lemma \ref{lem.inftydivergence} 
    to Example \ref{example_strong_not_stab}
    produces
    a strongly
    causal non-stably causal
    manifold with infinite but discontinuous distance.

    Applying Proposition \ref{prop.lenghtsupnomisscontinuous}
    to Example \ref{example_strong_not_stab}
    gives a
    strongly causal non-stably causal
    manifold with finite and continuous distance.
    The same construction as used in
    Examples \ref{example_fin_dis_causal} and
    \ref{example_nonsc_infinte_cts} can now be applied to the point
    $(0,-1)\in M$ of Example 
    \ref{example_strong_not_stab} to produce an example of
    a strongly causal non-stably causal
    manifold with infinite and continuous distance.
    
    This proves all the statements about strongly causal in Theorem
    \ref{thm:independence}.

    Examples \ref{example_missing},
    \ref{example.weaksingularity},
    and \ref{example.infinite} apply in this case too.

    \begin{example}[{\cite[Figure 9]{Minguzzi2008Causal}}]
      \label{example_strong_not_stab}
      Let $\tilde{N}=[-2,2]\times\R$.
      Let $L_1=\{(0,x)\in N: x\geq -1\}$, $L_2=\{(1,x)\in N:x\leq 1\}$ 
      and $L_3=\{(-1,x)\in N:x\leq 1\}$,
      ${N}=\tilde{N}\setminus\left(L_1\cup L_2\cup L_3\right)$.
      Let $M$ be the manifold given by identifying on $N$ the lines $\{2\}\times\R$ and 
      $\{-2\}\times\R$: see Figure \ref{fig-9}.
      That is, if $\sim$ is the equivalence relation given by identifying $(2, x)$ with 
      $(-2, x)$ for all $x\in\R$
      then $M=N/\sim$.
      The manifold $N$ carries the metric induced by inclusion into Minkowski
      space $\R^{1,1}$ and this metric induces a metric on
      $M$, which we denote by $g$.
      The Lorentzian metric $(M,g)$ is strongly causal but not stably causal. 
      Any small widening of the light cones will allow a closed time like
      curve to be created. The Lorentzian distance is finite and discontinuous on $M$.
    \end{example}

    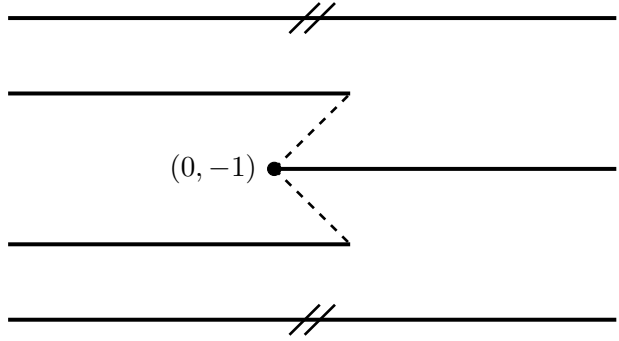
\begin{figure}
      \begin{center}
        \begin{tikzpicture}
          \draw[line width=1.5pt, smooth, samples=100, domain=-4:4, variable=\x] 
            plot (\x, 2);
          \draw[line width=1.5pt, smooth, samples=100, domain=-4:4, variable=\x] 
            plot (\x, -2);

          \draw[line width=1.0pt, smooth, samples=100, domain=-0.2:0.2, variable=\x] 
            plot (\x-0.1,\x+2);
          \draw[line width=1.0pt, smooth, samples=100, domain=-0.2:0.2, variable=\x] 
            plot (\x+0.1,\x+2);
          \draw[line width=1.0pt, smooth, samples=100, domain=-0.2:0.2, variable=\x] 
            plot (\x-0.1,\x-2);
          \draw[line width=1.0pt, smooth, samples=100, domain=-0.2:0.2, variable=\x] 
            plot (\x+0.1,\x-2);

          \draw[line width=1.5pt, smooth, samples=100, domain=-4:0.5, variable=\x] 
            plot (\x, 1);
          \draw[line width=1.5pt, smooth, samples=100, domain=-0.5:4, variable=\x] 
            plot (\x, 0);
          \draw[line width=1.5pt, smooth, samples=100, domain=-4:0.5, variable=\x] 
            plot (\x, -1);

          \filldraw (-0.5,-0) circle (2.5pt);
          \node at (-1.3,-0) {$(0, -1)$};

          \draw[line width=1pt, dashed, smooth, samples=100, domain=-0.5:0.5, variable=\x] 
            plot (\x,-\x-0.5);
          \draw[line width=1pt, dashed, smooth, samples=100, domain=-0.5:0.5, variable=\x] 
            plot (\x,\x+0.5);

        \end{tikzpicture}
      \end{center}
      \caption{An illustration of 
        Example \ref{example_strong_not_stab}.
        The thick lines without two cross hatches have been removed
        from the manifold. The thick lines with cross hatches have
        been identified. The dashed line is a null surface
        of geometric importance. 
        }
        \label{fig-9}
    \end{figure}

  \subsection{Stably causal}

    Example \ref{example_missing} is a stably causal and
    not causally continuous manifold with finite and 
    discontinuous Lorentzian distance. 
    
    Applying a conformal transformation as in Lemma \ref{lem.inftydivergence}
    to Example \ref{example_missing}
    will produce a manifold with infinite and discontinuous
    Lorentzian distance. 
    
    Next recall that a stably causal manifold
    is distinguishing. Thus a stably causal manifold is
    causally continuous if and only if
    for all points $x$ in
    the manifold we have $\missingp{x}=\missingf{x}=\EmptySet$,
    \cite[Page 59]{BeemEhrlichEasley1996}.
    Thus Proposition \ref{prop.missingeverydiscontin}
    implies that every stably causal and not causally continuous 
    manifold has discontinuous Lorentzian distance, proving part of the final
    statement of Theorem \ref{thm:independence}.
    
    Minkowski space with a point removed 
    has finite and continuous Lorentzian distance and
    is stably causal, causally continuous but not causally simple.
    
    The conformal transformations 
    used in 
    Examples \ref{example_fin_dis_causal} and
    \ref{example_nonsc_infinte_cts} can be applied to Minkowski
    space with a point removed to produce a stably causal,
    causally continuous not causally simple, manifold with infinite 
    and continuous Lorentzian distance.
    
    This proves all the statements about stable causal 
    in Theorem \ref{thm:independence}.

  \subsection{Causally continuous}

    Example 
    \ref{example.weaksingularity}
    gives a causally continuous not causally simple
    manifold with finite and discontinuous Lorentzian distance.

    Example \ref{example.infinite} presents a
    causally continuous not causally simple
    manifold with infinite and discontinuous Lorentzian distance.
    
    Minkowski space with a point removed 
    is causally continuous, not causally simple,
    with finite and continuous Lorentzian distance.
    
    Applying a conformal transformation as 
    Example \ref{example_fin_dis_causal} or
    Example \ref{example_nonsc_infinte_cts} 
    to Minkowski space with a point removed
    produces
    a causally continuous, not causally simple,
    with infinite and continuous Lorentzian distance.
    
    This proves all the statements about causally continuous in Theorem
    \ref{thm:independence}.

  \subsection{Causally simple}

    Example \ref{example_caus_not_gh_finte_cts}, below, gives
    a causally simple non-globally hyperbolic manifold
    with finite and continuous Lorentzian distance.

    Construction of a causally simple non-globally hyperbolic
    manifold with finite and discontinuous distance is complicated
    by the need to maintain causal simplicity.
    Note that the manifold $\widetilde{M}$ in Example \ref{example_caus_not_gh_finte_cts}
    is a submanifold of the manifold $M$
    in Example \ref{example.weaksingularity}. Restricting the
    metric $\phi_*g$ on $M$ to $\widetilde{M}$ produces
    a causally simple non-globally hyperbolic manifold with finite
    and discontinuous distance.

    Applying a conformal transformation as in 
    Lemma \ref{lem.inftydivergence} 
    to the manifold constructed in Example \ref{example_caus_not_gh_finte_cts}
    produces a causally simple non-globally hyperbolic
    manifold with infinite and discontinuous distance.
    
    Using the technique illustrated in Examples
     \ref{example_fin_dis_causal} or
    Example \ref{example_nonsc_infinte_cts} 
    on the manifold of Example \ref{example_caus_not_gh_finte_cts}
    gives a
    causally simple non-globally hyperbolic manifold
    with infinite and continuous Lorentzian distance.

    \begin{example}[{\cite[Figure 10]{Minguzzi2008Causal}}]
      \label{example_caus_not_gh_finte_cts}
      Let 
      $$
      \widetilde{M}=\R^2\setminus\big\{(x,y\in\R^2:x\leq -1\}\cup
        \{(x,y)\in\R^2:x\geq 0, \abs{y}\leq 2x\big\}
        $$ 
        with
      the metric $-dy^2+dy^2$ induced by the inclusion of $\widetilde{M}$ into
      $2$-dimensional Minkowski space: see Figure \ref{fig-10}.
      This is a causally simple non-globally hyperbolic manifold
      with finite and continuous Lorentzian distance.
    \end{example}

    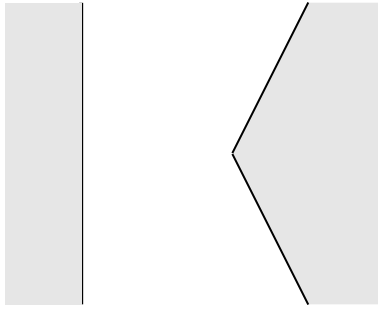
\begin{figure}
      \begin{center}
        \begin{tikzpicture}
          \draw[smooth, samples=100, domain=-2:2, variable=\x, line width=1.5pt] 
            plot (-2, \x);
          \draw[smooth, samples=100, domain=0:1, variable=\x, line width=1.5pt] 
            plot (\x, 2*\x);
          \draw[smooth, samples=100, domain=0:1, variable=\x, line width=1.5pt] 
            plot (\x, -2*\x);
          \fill[gray!20] 
            (1,2) -- (2,2) -- (2,-2) -- (1,-2) -- (0,0) -- cycle;
          \fill[gray!20] 
            (-2,2) -- (-3,2) -- (-3,-2) -- (-2,-2) -- cycle;
        \end{tikzpicture}
      \end{center}
      \caption{An illustration of 
        Example \ref{example_caus_not_gh_finte_cts}.
        The black lines and grey regions have been removed from 
        $\R^2$.
        }
        \label{fig-10}
    \end{figure}

  \subsection{Globally hyperbolic}
    The Lorentzian distance is necessarily finite and continuous
    in a globally hyperbolic manifold,
    \cite[Lemma 4.5]{BeemEhrlichEasley1996}.

\appendix
\section{Appendix - a limit curve theorem}\label{appendix:limitcurves}

  This appendix collects a few important details about limit curves together.
  In contrast to existing results,
  we emphasise the case of continuous causal curves. 
  In the differentiable case the results below are classical,
  see \cite{BeemEhrlichEasley1996, Minguzzi2008, Penrose1972}.
  Where possible we cite the related results.

  \begin{definition}[{\cite[Page 54]{BeemEhrlichEasley1996}}]
    A continuous curve $\gamma:(a,b)\to M$, $a,b\in\R$, $a<b$, is
    a
    future directed causal curve if for all $t\in (a,b)$ there
    exists $\epsilon>0$ and a convex normal neighbourhood, $U$, of
    $\gamma(t)$ with $\gamma(t-\epsilon,t+\epsilon)\subset U$ so that
    for any $t_1,t_2\in(t-\epsilon,t+\epsilon)$,
    $t_1<t_2$, there is a smooth
    future directed causal curve lying in $U$ from
    $\gamma(t_1)$ to $\gamma(t_2)$.
  \end{definition}

  Continuous causal curves are discussed in some detail in
  \cite[Definition 2.3 ff]{whale2015generalizations}.

  \begin{lemma}
    Every continuous causal curve has
    a parameterisation with respect to
    which it is locally Lipschitz and
    thus differentiable almost everywhere.
  \end{lemma}
  \begin{proof}
    The proof can be found in \cite[Pages 75 and 76]{BeemEhrlichEasley1996}.
    It uses the Lorentzian metric to show that every continuous
    causal curve is locally Lipschitz.
  \end{proof}


  \begin{lemma}[{\cite[Lemma 5.9]{ONeill1983}}]
    \label{lem.smoothdelta}
    If $U\subset M$ is a open convex set,
    then the function $\Delta:U\times U\to TM$ defined by
    $\Delta(x,y):=\exp_x^{-1}(y)\in T_xM$ is smooth.
  \end{lemma}

  \begin{definition}[{\cite[Definition 2.13]{Penrose1972}}]
    Let $U\subset M$ is be an open convex set,
    define $\Phi(x,y)=g(\Delta(x,y),\Delta(x,y))$.
  \end{definition}

  \begin{lemma}\label{lem.convergenceoflengthsofgeodesics}
    Let $U\subset M$ be an open convex normal neighbourhood and for all 
    $u,v\in U$ let $\gamma_{uv}$ be the unique geodesic
    from $u$ to $v$ in $U$.
    If $(x_i)_{i\in\N}$ is a sequence in $U$ converging to $x\in U$
    and 
    $(y_i)_{i\in\N}$ is a sequence in $U$ converging to $y\in U$
    and for all $i\in\N$, $y_i\in\future{x_i}$ then
    $\lim_{i\to\infty}L\left(\gamma_{x_iy_i}\right)=L\left(\gamma_{xy}\right)$.
  \end{lemma}
  \begin{proof}
    We can always parametrise $\gamma_{x_iy_i}$ so that
    $\gamma_{x_iy_i}:[0,1]\to M$. In this case, by definition,
    $\gamma'_{x_iy_i}(1)=\Delta(x_i,y_i)$ and we can see that
    $L(\gamma_{x_iy_i})=\sqrt{-\Phi(x_i,y_i)}$.
    Hence $L(\gamma_{x_iy_i})$ depends continuously on
    $x_i$, $y_i$ and so $\lim_{i\to\infty}L(\gamma_{x_iy_i})=L(\gamma_{xy})$.
  \end{proof}

  A slightly more nuanced approach, handling the cases $\Phi(x_i,y_i)=0$
  and $\Phi(x_i,y_i)\neq 0$ for all $i\in\N$ would give a little additional
  insight into the differential dependence of 
  $L(\gamma_{x_iy_i})$ on $x_i,y_i$. We shall not need this, however.

  \begin{definition}[see {\cite[Definitions 7.1 and 7.4]{Penrose1972}}]
    Let $\gamma:[a,b]\to M$, $a,b\in\R$, $a<b$, 
    be a continuous causal curve. A partition of $[a,b]$
    is a finite subset $\{t_i\in[a,b]:i=1,\ldots m\}$
    so that $t_1=a<t_2<\cdots<t_{m-1}<t_m=b$
    and for all $i=1,\ldots,m-1$ there
    exists an open convex normal neighbourhood containing
    $\gamma(t_i)$ and $\gamma(t_{i+1})$.
    Let $\Xi(\gamma)$ denote the set of
    all such partitions of $I$. Note that $\Xi(\gamma)$ depends not
    only on the domain of $\gamma$ but also on $\gamma$ due to 
    the requirement that for all $i$ in a partition
    $\gamma(t_i)$ and $\gamma(t_{i+1})$ are in a common open
    convex normal neighbourhood.

    Define the length of $\xi\in\Xi(\gamma)$ as
    \[
      L(\xi,\gamma)=\sum_{i=1}^{n-1}
        L\left(\gamma_{\gamma(t_i),\gamma(t_{i+1})}\right),
    \]
    where $\gamma_{\gamma(t_i),\gamma(t_{i+1})}$ is the, now awkward
    expression for the, unique geodesic from $\gamma(t_i)$ to $\gamma(t_{i+1})$
    lying in the assumed convex normal neighbourhood.
    Define the length of
    $\gamma$ as
    \[
      L(\gamma)=\inf\left\{L(\xi,\gamma):\xi\in\Xi(\gamma)\right\},
    \]
    where, in an abuse of notation,
    we overloaded the symbol $L$.
    If the domain of $\gamma:I\to\R$ is not compact then we
    define
    \[
      L(\gamma)=\sup\{L(\gamma|_K):K\subset I\ \text{is compact}\},
    \]
    once again abusing the symbol $L$.

    Since geodesics are length maximising if $\xi'\subset\xi\in\Xi(\gamma)$ then
    $L(\xi)<L(\xi')$ and if $\xi''=\xi\cup\xi$, 
    $\xi,\xi',\xi''\in\Xi(\gamma)$ then
    $L(\xi'')\leq\min\{L(\xi),L(\xi')\}$. 
    As $L(\xi)\geq 0$ for all
    $\xi\in\Xi(\gamma)$ the length of $\gamma$ is well defined and
    finite for curves with compact domain.
  \end{definition}

  \begin{lemma}[see {\cite[Definition 7.4]{Penrose1972}}]
    \label{lem.lengthdefinitionsareequal}
    Let $\gamma:I\to\R$ be a future directed locally Lipschitz
    causal curve then
    \[
      \int_I\sqrt{-g(\gamma',\gamma')}\dd t=
        \inf\{L(\xi,\gamma):\xi\in\Xi(\gamma)\}
      =\sup\{L(\gamma|_K):K\subset I\ \text{is compact}\}.
    \]
  \end{lemma}
  \begin{proof}
    This follows from standard results regarding the relationship
    of rectifiable curves, local Lipschitz continuity and
    path integrals, \cite{RectifiableCurve}.
  \end{proof}

  \begin{lemma}[{see \cite[Theorem 2.4]{Minguzzi2008}
    or \cite[Proposition 8.2]{BeemEhrlichEasley1996}}]
    \label{lem.uniformconvergenceimplieslengthconvergence}
    Let $(M,h)$ be a Riemannian manifold and
    let $b\in\R\cup\{\infty\}$.
    For each $i\in\N$,
    let $\gamma_i:[0,b)\to M$ be a future directed causal curve.
    If there exists a continuous causal curve
    $\gamma:[0,b)\to M$ so that 
    the sequence $(\gamma_i)_{i\in\N}$ converges to $\gamma$
    uniformly on compact subsets of $[0,b)$, with respect to the
    distance induced by $h$,
    then $\lim_{i\to\infty}L(\gamma_i)\leq L(\gamma)$.
  \end{lemma}
  
  
  \begin{proof}
    Let $K\subset I$ be compact.
    We will show that for all
    $\xi\in\Xi(\gamma|_K)$, with $m$ elements, there exists $N\in\N$
    so that $i\geq N$ implies that there exists $\xi_i\in\Xi(\gamma_i|_K)$
    so that $\xi=\xi_i$ and
    for all $i=1,\ldots,m-1$ the geodesics
    $\gamma_{\gamma_i(t_i),\gamma_i(t_{i+1})}$
    and $\gamma_{\gamma(t_i),\gamma_i(t_{i+1})}$
    are in the same convex normal neighbourhood.

    Choose $\xi\in\Xi(\gamma)$ and assume that $\xi$ has $m$ elements.
    For each $j=1,\ldots,m-1$ choose $U_j$ a convex normal neighbourhood
    so that $\gamma(t_j),\gamma(t_{j+1})\in U_j$.
    Since there are only a finite number of $U_j$ there exists $\epsilon>0$
    so that for all $j=1,\ldots, m-1$
    the ball based at $t_j$ of radius $\epsilon$ is
    contained in $U_j$ and $U_{j+1}$. 
    Since $K$ is compact, by assumption there exists $N\in\N$
    so that $i\geq N$ implies that for all $t\in K$, 
    $d(\gamma_i(t),\gamma(t))<\epsilon$, where
    $d$ is the \emph{Riemannian} distance induced by $h$.
    In particular for each $j=1,\ldots,m-1$,
    $i>N$ implies that $\gamma_i(t_j)\in U_j\cap U_{j+1}$.
    This implies that for all $j=1,\ldots,m-1$, $i\geq N$ 
    $\gamma_{\gamma_i(t_j),\gamma_i(t_{j+1})}\subset U_j$.
    Thus $\xi\in\Xi(\gamma_i)$ as required.

    Since uniform convergence on compact subsets implies pointwise convergence,
    Lemma \ref{lem.convergenceoflengthsofgeodesics}
    implies that 
    for all $\epsilon>0$ and all $j=1,\ldots,m-1$ there exists
    $N(\epsilon, j)\in\N$ so that $i\geq N(\epsilon, j)$ implies that
    $L(\gamma_{{\gamma_i(t_j)\gamma_i(t_{j+1})}})<
     L(\gamma_{\gamma(t_j)\gamma(t_{j+1})}) + \epsilon$.
    Since $j=1,\ldots,m-1$ this implies that for all $\epsilon>0$
    there exists $N(\epsilon)\in\N$ so that $i\geq N$ implies that
    $\len{\xi_i,\gamma|_K}<\len{\xi,\gamma|_K}+\epsilon$.
    This implies that $\len{\gamma_i|_K}\leq\len{\gamma|_K}+\epsilon$.
    Taking supremum over $K$ we get that
    for all $\epsilon>0$ there exists $N(\epsilon)\in \N$ so that
    $i\geq N(\epsilon)$ implies that
    $\len{\gamma_i}\leq \len{\gamma}+\epsilon$.
    This implies that $\lim_{i\to\infty}L(\gamma_i)\leq L(\gamma)$
    as required.
  \end{proof}

  \begin{lemma}[{see \cite[Lemma 2.7]{Minguzzi2008} or 
    \cite[Proposition 3.31]{BeemEhrlichEasley1996}}]
    \label{lem.pointwiseconvergenceforcausalcharacter}
    Let $b\in\R\cup\{\infty\}$.
    For each $i\in\N$,
    let $\gamma_i:(a,b)\to M$ be a future directed causal curve.
    If there exists a continuous curve
    $\gamma:(a,b)\to M$ so that 
    the sequence $(\gamma_i)_{i\in\N}$ converges pointwise
    to $\gamma$, then $\gamma$ is a future directed causal curve.
  \end{lemma}
  \begin{proof}
    Let $t\in(a,b)$ and choose $U$ an open convex normal neighbourhood
    containing $\gamma(t)$. Since $U$ is open there exists
    $\epsilon>0$ so that $\gamma(t-\epsilon,t+\epsilon)\subset U$.
    Let $t_1,t_2\in(t-\epsilon,t+\epsilon)$, $t_1<t_2$.
    By assumption there exists $N\in\N$ so that $i\geq N$
    implies that $\gamma_i(t_1),\gamma_i(t_2)\in U$.
    Without loss of generality we can assume that $N=0$.

    For each $i\in\N$ as each $\gamma_i$ is future directed causal 
    and as $U$ is convex normal there exists a future directed
    causal geodesic in $U$ from $\gamma_i(t_1)$ to $\gamma_i(t_2)$.
    Let $v_i=\Delta(\gamma_i(t_1),\gamma_i(t_2))$
    then $(v_i)_{i\in\N}$ converges to $v=\Delta(\gamma(t_1),\gamma(t_2))$,
    by Lemma \ref{lem.smoothdelta}.
    Since each $\gamma_i$ is future directed and timelike
    $g(T,v_i)\leq 0$ and $g(v_i,v_i)\leq 0$.
    Taking the limit with respect to $i$ shows that
    $g(T,v)\leq 0$ and $g(v,v)\leq 0$.
    Hence $v\in T_{\gamma(t_1)}M$ is future directed and causal.
    By construction $\exp_{\gamma(t_1)}(v)=\gamma(t_2)$.
    Thus as $U$ is convex normal the unique geodesic between
    $\gamma(t_1)$ and $\gamma(t_2)$ is the curve
    $t\mapsto\exp_{\gamma(t_1)}(tv)$ which is future directed
    and causal as required.
  \end{proof}

  Like all limit curve results the lemma below is
  based on Arzel\`a's Theorem,
  \cite[Theorem 3.30]{BeemEhrlichEasley1996}. Ours is, essentially,
  a more precise form of 
  \cite[Proposition 3.31]{BeemEhrlichEasley1996} and
  \cite[Lemma 6.2.1]{HawkingEllis1975}. For a detailed study of
  limit curve theorems refer to \cite{Minguzzi2008}.

  \begin{lemma}\label{CurLim:Lem.CurveUniformConvergenceInBoundedRegion}
    Let ${M}$ be a manifold and let
    $d:{M}\times{M}\to\mathbb{R}$,
    be the distance induced by a complete Riemannian metric.
    Let $B$ be a bounded
    subset of ${M}$. Let $\gamma_i:I_i\to{M}$ be a sequence of
    $C^0$ curves in ${M}$,
    so that for some $N\in\N$,
    $n>N$ implies that
    $\gamma_n\subset\overline{B}$, 
    where $I_i=[0,b_i]$, 
    $b_i\in\mathbb{R}$ or $[0,\infty)$ 
    in which case we set $b_i=\infty$. Furthermore, 
    we assume that each $\gamma_i$ is 
    parametrised so that
    for all $i\in\N$ and for all $t_1,t_2\in I_i$,
    \[
    	d(\gamma_i(t_1),\gamma_i(t_2))\leq |t_1-t_2|.
    \]
    
    Let $b=\sup b_i$; if $b=\infty$ let $Y_i=[0,b_i)$ and $X=[0,\infty)$, 
    otherwise let
    $Y_i=[0,b_i]$ and $X=[0,b]$.    
    Then there
    exists a sequence of strictly monotonic increasing, bijective, 
    smooth changes of parameter $f_i:X\to Y_i$, so that
    there is a subsequence of
    $(\gamma_i\circ f_i)_{i\in\N}$ that converges uniformly, 
    on compact subsets of $X$, to a $C^0$ curve
    $\gamma:X\to{M}$,
    which lies in $\overline{B}$.
  \end{lemma}
  \begin{proof}
    We have three cases to
    consider.  If $b=\infty$ and $b_i=\infty$, let $f_i(x)=x$. 
    If $b=\infty$ and $b_i\neq\infty$, let
    $f_i(x)=\frac{2b_i}{\pi}\text{arctan}\left(\frac{\pi x}{2 b_i}\right)$.  
    Otherwise $b<\infty$ and we let
    $f_i(x)=\frac{b_ix}{b}$.
    In any case, we know that $f_i:X\to Y_i$ is a strictly monotonic, 
    increasing, bijective smooth function.
    
    We now show that, in any case, we have the relation  $f_i(u)-f_i(v)\leq u-v$
    and therefore that $\{f_i:i\in\N\}$ is a uniformly equicontinuous family.
    When $b<\infty$ we know that
    \begin{align*}
    	f_i(u)-f_i(v)&=\frac{b_i}{b}u-\frac{b_i}{b}v
    		=\frac{b_i}{b}(u-v)
    		\leq u-v.
    \end{align*}
   	When $b=\infty$ and $b_i=\infty$ we know that
   	$f_i(u)-f_i(v)=u-v\leq u-v$.
   	When $b=\infty$ and $b_i<\infty$ we note that
   	\begin{align*}
   		\frac{d}{dx}f_i&=\frac{1}{1+\left(\frac{\pi x}{2b_i}\right)^2}
   			\leq 1.
   	\end{align*}
   	Let $g(x)=x$, then as $f_i(0)=0$, $g(0)=0$ and
   	$0<\frac{d}{dx}f_i<\frac{d}{dx}g$ for all $x>0$, we know that
   	$f_i(x)<g(x)$ for all $x>0$. Therefore we have that
   	$f_i(u)-f_i(v)\leq g(u)-g(v)=u-v$,
    as required.

    We now show that, in any case, 
    $\{\tilde{\gamma}_i=\gamma_i\circ f_i:i\in\N\}$ 
    is a uniformly equicontinuous family.
    By assumption for all $i\in\N$ and all $t_1,t_2\in I_i$ we know that  
    $d(\gamma_i(t_1),\gamma_i(t_2))\leq|t_1-t_2|$. 
    Thus for all $t_1,t_2\in X$ we know that
    \begin{align*}
      d(\tilde{\gamma}_i(t_1),\tilde{\gamma}_i(t_2))&=
        d(\gamma_i\circ f_i(t_1),\gamma_i\circ f_i(t_2))\\
        &\leq|f_i(t_1)-f_i(t_2)|\\
        &\leq |t_1-t_2|.
    \end{align*}
		Since $|t_1-t_2|$ does not depend on $i$, 
    the collection of functions 
    $\gamma_i\circ f_i=\tilde{\gamma}_i:X\to{M}$ is 
		uniformly equicontinuous.

    We now show that $\{\tilde{\gamma}_i(t):i\in\N\}$ is
    bounded for each $t\in X$. Let $t\in X$ and let
    $x_i=\tilde{\gamma}_i(t)$. Since
    $\tilde{\gamma}_i(X)=\gamma_i(Y_i)$, we can see that
    for all $n>N$, $x_n\in \overline{B}$, by assumption. The set
    $X_{N}=\{x_i:i\leq N\}$ is finite and because
    $\overline{B}$ is bounded there must exist $B\in\mathbb{R}^+$ so
    that $d(x_i,x_j)<B$ for all $i,j$. Hence
    $\{\tilde{\gamma}_i(t):i\in\N\}$ is bounded for each $t\in X$.
    So, by Arzel\`a's theorem \cite[Theorem 3.30]{BeemEhrlichEasley1996}, 
    there exists some $C^0$ curve $\gamma:X\to{M}$ such
    that there is a subsequence of $(\tilde{\gamma}_i)_{i\in\N}$ which
    converges uniformly to $\gamma$ on compact subsets of $X$.  
    
    To show that $\gamma\subset\overline{B}$ we must show that
    for all $t\in X$, $\gamma(t)\in\overline{B}$.  
    As there is a subsequence $(\tilde{\gamma}_{k_i})_{i\in\N}$ of
    $(\tilde{\gamma}_i)$ that converges to 
    $\gamma$ uniformly on compact subsets of $X$ and as $[t,t+\epsilon]$
    is a compact subset of $X$, for some
    $\epsilon>0$, we can conclude that 
    $\tilde{\gamma}_{k_i}(t)\to \gamma(t)$.  We know, however, that
    for all $n>N$, $\tilde{\gamma}_n(t)\in\overline{B}$, 
    thus there exists some $m_0\in 
    \N$ so that for all $i>m_0$, $k_i>N$
    and therefore $\tilde{\gamma}_{k_i}(t)\in\overline{B}$.  
    Hence $\gamma(t)\in\overline{B}$ as required.
  \end{proof}

\end{document}